\documentclass[12pt]{amsart}

\usepackage[all]{xypic}
\usepackage{tikz}
\usetikzlibrary{arrows} 
\usetikzlibrary{decorations.markings}
\usepackage{graphicx}
\usepackage{bm}
\usepackage{epsf}
\usepackage{verbatim} 
\usepackage{colonequals} 
\usepackage{amsmath}
\usepackage{amsfonts}
\usepackage{amssymb}
\usepackage{mathrsfs}
\usepackage{amsthm}
\usepackage{newlfont}

\theoremstyle{plain}
\newtheorem{thm}{Theorem}[section]
\newtheorem{prop}[thm]{Proposition}
\newtheorem{lem}[thm]{Lemma}

\newtheorem{theoremvar}{Conjecture}
\newenvironment{conj*}[1]
{\renewcommand{\thetheoremvar}{#1}\theoremvar}
{\endtheoremvar}

\theoremstyle{definition}
\newtheorem{defn}[thm]{Definition}

\newtheorem{example}[thm]{Example}

\theoremstyle{remark}
\newtheorem{rem}[thm]{Remark}

\numberwithin{equation}{section}
\numberwithin{thm}{section}

\DeclareMathOperator{\Hom}{Hom}
\DeclareMathOperator{\Har}{Har}
\DeclareMathOperator{\End}{End}
\DeclareMathOperator{\Aut}{Aut}

\DeclareMathOperator{\Ht}{ht}
\DeclareMathOperator{\Frob}{Frob}

\DeclareMathOperator{\Spec}{Spec}

\DeclareMathOperator{\chr}{char}

\DeclareMathOperator{\ord}{ord}

\DeclareMathOperator{\Gal}{Gal}
\DeclareMathOperator{\GL}{GL}
\DeclareMathOperator{\PGL}{PGL}
\DeclareMathOperator{\SL}{SL}

\DeclareMathOperator{\Div}{Div}

\DeclareMathOperator{\Stab}{Stab}

\newcommand{\sep}{\mathrm{sep}}
\newcommand{\alg}{\mathrm{alg}}

\newcommand{\tor}{\mathrm{tor}}
\newcommand{\new}{\mathrm{new}}

\newcommand{\dd}{\mathrm{d}}

\newcommand{\Ver}{\mathrm{Ver}}
\newcommand{\Ed}{\mathrm{Ed}}

\newcommand{\fE}{\mathfrak{E}}

\newcommand{\fM}{\mathfrak{M}}

\newcommand{\fP}{\mathfrak{P}}

\newcommand{\fl}{\mathfrak{l}}
\newcommand{\fm}{\mathfrak{m}}
\newcommand{\fn}{\mathfrak{n}}

\newcommand{\fp}{\mathfrak{p}}
\newcommand{\fq}{\mathfrak{q}}

\newcommand{\cB}{\mathcal{B}}
\newcommand{\cC}{\mathcal{C}}

\newcommand{\cE}{\mathcal{E}}

\newcommand{\cG}{\mathcal{G}}
\renewcommand{\cH}{\mathcal{H}}
\newcommand{\cI}{\mathcal{I}}
\newcommand{\cJ}{\mathcal{J}}

\renewcommand{\cL}{\mathcal{L}}
\newcommand{\cM}{\mathcal{M}}

\newcommand{\cO}{\mathcal{O}}

\newcommand{\cS}{\mathcal{S}}
\newcommand{\cT}{\mathcal{T}}
\newcommand{\cU}{\mathcal{U}}

\newcommand{\cY}{\mathcal{Y}}

\newcommand{\sB}{\mathscr{B}}

\newcommand{\sT}{\mathscr{T}}

\newcommand{\A}{\mathbb{A}}

\newcommand{\C}{\mathbb{C}}

\newcommand{\F}{\mathbb{F}}
\newcommand{\gm}{\mathbb{G}}

\newcommand{\N}{\mathbb{N}}

\newcommand{\p}{\mathbb{P}}
\newcommand{\Q}{\mathbb{Q}}
\newcommand{\R}{\mathbb{R}}

\newcommand{\T}{\mathbb{T}}

\newcommand{\Z}{\mathbb{Z}}

\newcommand{\boz}{\boldsymbol{z}}

\newcommand{\G}{\Gamma}
\newcommand{\To}{\longrightarrow}
\newcommand{\bs}{\setminus}
\newcommand{\Fi}{F_\infty}
\newcommand{\Ci}{\C_\infty}

\newcommand{\oF}{\overline{\F}}
\newcommand{\oK}{\overline{K}}

\newcommand{\La}{\Lambda}
\newcommand{\la}{\lambda}

\newcommand{\twist}[1]{#1\!\left\{\tau\right\}}

\newcommand{\atwist}[2]{#1\left\{#2\right\}}

\newcommand{\ls}[2]{#1\!\left(\mkern-4mu\left(#2\right)\mkern-4mu\right)}

\newcommand{\Mod}[1]{\ (\mathrm{mod}\ #1)}
\newcommand{\abs}[1]{\left|#1\right|}

\newcommand{\longhookrightarrow}{\lhook\joinrel\longrightarrow}

\DeclareMathOperator{\Cot}{\mathrm{Cot}} 
\DeclareMathOperator{\spe}{\mathrm{spe}} 

\begin{document}
\title{Ogg's conjectures over function fields}

\author{C\'ecile Armana}
\address{Universit\'{e} de Franche-Comté, CNRS, LmB (UMR 6623), F-25000 Besan\c{c}on, France}
\email{cecile.armana@univ-fcomte.fr}

\author{Sheng-Yang Kevin Ho}
\address{Department of Mathematics, Pennsylvania State University, University Park, Pennsylvania, United States of America}
\email{kevinho@psu.edu}

\author{Mihran Papikian}
\address{Department of Mathematics, Pennsylvania State University, University Park, Pennsylvania, United States of America}
\email{papikian@psu.edu}

\thanks{The first author was supported by ANR -- France (French National Research Agency), PadLEfAn project ANR-22-CE40-0013. The third author was supported in part by a grant from the Simons Foundation, Award No. MPS-TSM-00008093.}

\subjclass[2010]{11G09, 11G05}
\keywords{Drinfeld modules; Elliptic curves}

\dedicatory{Dedicated to Andrew Ogg}



\maketitle


\section{Introduction} 

In the early 1970s, Andrew Ogg made several conjectures about the rational torsion points of elliptic curves over $\Q$ and 
the Jacobians of modular curves; cf. \cite{Ogg1970},  \cite{Ogg1973}, \cite{Ogg1975}. In particular, he proposed a finite list 
of groups that can occur as rational torsion subgroups of elliptic curves over $\Q$. These conjectures were proved shortly after by Barry Mazur \cite{Mazur1977}, \cite{Mazur1978} as a consequence of his fundamental study of the arithmetic properties of modular curves and Hecke algebras. The powerful techniques introduced by Mazur had tremendous impact on later developments in arithmetic geometry; for example, they have been instrumental in the proof of the Main Conjecture of Iwasawa theory \cite{MW} and the proof of Fermat's Last Theorem \cite{Wiles}, \cite{Ribet1990}.

At around the same time, Vladimir Drinfeld introduced certain function field analogues of elliptic curves, nowadays called \textit{Drinfeld modules}. 
In fact, due to this analogy, in his seminal paper \cite{Drinfeld} 
Drinfeld called these objects ``elliptic modules''. 
Drinfeld's motivation was to construct function field analogues 
of classical modular curves classifying elliptic curves with some additional data, and use these to relate automorphic forms 
and Galois representations, in line with the program envisioned by Langlands. 
The theory of Drinfeld modules and their generalizations, called \textit{shtukas}, has since lead to a successful resolution 
of the Langlands correspondence 
over function fields, cf. \cite{DrinfeldICM}, \cite{LRS}, \cite{Lafforgue}, \cite{VLafforgue}, and   
it continues to play a central role in number theory because of its applications to many other important problems, 
such as the Birch and Swinnerton-Dyer conjecture, cf. \cite{YunZhang}. 

Due to the close similarity between elliptic curves and Drinfeld modules, Ogg conjectures have natural analogues in the function field setting, as was already suggested by Mazur in \cite{Mazur1977}. 
As in the classical case, a comprehensive approach to these conjectures heavily relies on the 
theory of (Drinfeld) modular forms and modular varieties. Since some of the necessary aspects of this theory were developed only in the 1990s, the analogues of Ogg's conjectures were stated and proved in certain cases only in the 2000s; cf. \cite{SchweizerMZ}, \cite{PalDoc}, \cite{ArmanaANT}. 

In this paper, we review the function field analogues of Ogg's conjectures, their current status, and the methods that have been applied to prove some of these conjectures. The methods are 
based on the ideas of Mazur and Ogg, but there are interesting differences and technical complications 
that arise in the function field setting, as well as intriguing possible new directions for generalizations. 

The contents of the paper are as follows. 
In Section \ref{sOggConj}, we recall the statements of Ogg's original conjectures, some of their generalizations, and what is currently known about these conjectures (Ogg's conjectures have been generalized in many different directions, and our exposition of these generalizations is by no means exhaustive). 
In Section \ref{sDM}, we review the basic theory 
of Drinfeld modules, putting the emphasis on those aspects of the theory that are relevant for Ogg's conjectures. In Section \ref{sDMF}, we review the theory of Drinfeld modular forms, which is 
extensively used in the proofs of the analogues of Ogg's conjectures.  
In Sections \ref{sDMT} and \ref{sDJT=C}, we give the 
statements of Ogg's conjectures in the setting of Drinfeld's theory, the current status of these conjectures, a brief summary of the ideas that go into the proofs, and some of the nontrivial 
complications that arise in this setting. 


\section{Ogg's conjectures}\label{sOggConj}

\subsection{Torsion of elliptic curves} 
Let $N$ be a positive integer and let $\G_0(N)$ (resp. $\G_1(N)$) be the congruence subgroup 
of $\SL_2(\Z)$ consisting of matrices that are upper-triangular (resp. are unipotent) modulo $N$. 
Let $Y_i(N)(\C):=\G_i(N)\bs \cH$, $i=0,1$, be the affine modular curve, where $\cH$ is the complex upper half-plane and $\SL_2(\Z)$ acts on $\cH$ by linear fractional transformations. Denote by $X_i(N)(\C)$ the compactification of $Y_i(N)(\C)$ obtained by adjoining finitely many points, called \textit{cusps}. 
The curve $X_i(N)(\C)$ has a canonical non-singular projective model $X_i(N)$ defined over $\Q$, and in this model the set of cusps is invariant under the action of $\Gal(\overline{\Q}/\Q)$. 

The first, and probably most famous, conjecture in Ogg's paper \cite{Ogg1975} 
gives a classification of possible torsion subgroups $E(\Q)_\tor$ of elliptic curves over $\Q$. 
This conjecture was essentially formulated by Beppo Levi in 1908, although Ogg was not aware of this as  Levi's work on the arithmetic of elliptic curves did not receive the attention it deserved. 
In any case, Ogg's papers were instrumental in spelling out the close connection between this problem 
and the theory of modular curves and in popularizing the conjecture.  

\begin{conj*}{TEC}\label{conjOgg1}
If $E$ is an elliptic curve 
over $\Q$ then its rational torsion subgroup $E(\Q)_\tor$ is one of the following fifteen groups:
$$
\Z/N\Z, \quad 1\leq N\leq 10 \text{ or } N=12;
$$
$$
\Z/2\Z\times \Z/2N\Z, \quad 1\leq N\leq 4. 
$$
This is essentially equivalent to saying that the modular curve 
$X_1(N)$ has no $\Q$-rational points besides cusps, unless $1\leq N\leq 10$ and $N=12$. 
\end{conj*}

A related conjecture, stated in \cite{Ogg1975}  as a problem, concerns rational cyclic subgroups of elliptic curves over $\Q$. If $C$ is a finite subgroup of an elliptic curve $E$ over $\Q$, then $C$ is said to be rational if $\sigma(C)=C$ for all $\sigma\in \Gal(\overline{\Q}/\Q)$.  

\begin{conj*}{TEC$^+$}\label{conjOgg2}
A rational cyclic subgroup of an elliptic curve over $\Q$ has order $1\leq N\leq 19$, or $N=21, 25, 27, 37, 43, 67, 163$. Equivalently, the modular curve $X_0(N)$ has no $\Q$-rational points besides cusps, unless $N$ is one of the listed values. 
\end{conj*}

The philosophy behind Conjectures \ref{conjOgg1} and \ref{conjOgg2} is that ``modular curves only have rational points for which there is a reason", the reason being geometric; cf. \cite[p. 17]{Ogg1975}.
More specifically, $Y_1(N)$ has a $\Q$-rational point if and only if $X_1(N)$ has genus $0$. The 
genus of $X_1(N)$ is $0$ exactly for $1\leq N\leq 10$ and $N=12$. Moreover, for these values $X_1(N)$ has infinitely many $\Q$-rational points, so the groups listed in Conjectures \ref{conjOgg1}  
occur as the rational torsion subgroups of infinitely many non-isomorphic elliptic curves over $\Q$. 
Similarly, $X_0(N)(\Q)$ is infinite if and only if its genus is $0$, which happens exactly 
for $1\leq N\leq 10$, and $N=12,13,16,18,25$. The curve $X_0(N)$ has genus $1$ and $Y_0(N)(\Q)$ 
is finite and known for $N=11,14,15,17,19,21,27$. The cases when $X_0(N)$ has genus $\geq 2$ and 
$Y_0(N)(\Q)$ is non-empty are $N=37, 43, 67, 163$. 
For $N=43, 67, 163$, the curve $Y_0(N)$ has one $\Q$-rational point which corresponds to an elliptic curve over $\Q$ with CM by the ring of integers of $\Q(\sqrt{-N})$ (note that in these cases the class number of $\Q(\sqrt{-N})$ is one). The case $N=37$ is somewhat special: $Y_0(37)$ has two $\Q$-rational points, whose existence is related to the fact that the hyperelliptic involution of $X_0(37)$ is not the Atkin-Lehner involution. 
 	
Conjectures \ref{conjOgg1} and \ref{conjOgg2} were proved by Mazur in \cite{Mazur1977}, \cite{Mazur1978}, where he developed an intricate arithmetic theory of modular curves, Eisenstein ideals in Hecke algebras, and isogeny characters of elliptic curves. 

By extending Mazur's techniques, one can attack the problem of classifying the points on $X_1(N)$ and $X_0(N)$ that are rational over 
number fields of given degree; cf. \cite{Kamienny}, \cite{DKSS}. One can also 
consider rational points on Shimura curves classifying abelian surfaces with 
quaternionic multiplication; cf. \cite{Jordan}, \cite{AM}. 

The Uniform Boundlessness Conjecture (UBC) extends Conjecture \ref{conjOgg1} to all number fields, but without giving an explicit classification. This conjecture is mentioned in \cite{Ogg1975} as a ``folklore conjecture": 

\begin{conj*}{UBC}\label{conjUBC}
	The orders of rational torsion subgroups of elliptic curves over a number field $K$ are uniformly bounded by a constant depending only on the degree $[K:\Q]$.  
\end{conj*}

The \ref{conjUBC} was proved by Merel \cite{Merel}, again building upon Mazur's work and subsequent refinements by Kamienny, and combining them with the recent at the time progress on the Birch and Swinnerton-Dyer conjecture.  

Finally, one can extend Conjectures \ref{conjOgg1} and \ref{conjOgg2} to modular curves classifying elliptic curves with 
other level structures, such as split or non-split Cartan subgroups. These conjectures are closely related to Serre's Uniformity Question (SUQ): 

\begin{conj*}{SUQ}\label{conjSUQ}
	The Galois representation $\rho_{E, p}\colon \Gal(\overline{\Q}/\Q)\to \Aut(E[p])\cong \GL_2(\F_p)$, arising from the action 
	of the Galois group on the $p$-torsion subgroup of an elliptic curve over~$\Q$ without complex multiplication, is surjective once $p>37$. 
\end{conj*}

Although the \ref{conjSUQ} remains open, substantial progress has been made: it is known that for
$p > 37$ the representation $\rho_{E, p}$ is either surjective or has image contained in the normaliser of a non-split Cartan subgroup; cf. \cite{BP}, \cite{BPR2013}, \cite{BDM}. 


\subsection{Torsion of the Jacobian variety of $X_0(N)$}  
Let $J_0(N)$ be the Jacobian variety of the modular curve $X_0(N)$. 
Let $\cC_N$ be the subgroup of $J_0(N)$ generated by the divisor classes $[c]-[c']$, where $[c]$ and $[c']$ run over the cusps of $X_0(N)$. By a result of Manin and Drinfeld \cite{DrinfeldTwoThms}, $\cC_N$ is a finite group. Denote by $\cC_N(\Q)$ the subgroup of $\cC_N$ fixed by the Galois action, so $\cC_N(\Q)\subseteq J_0(N)(\Q)_\tor$. Let $\cC(N)$ be the \textit{rational cuspidal divisor class group} of $X_0(N)$; this is the subgroup of $\cC_N(\Q)$ generated by the linear equivalence classes of degree $0$ rational cuspidal divisors on $X_0(N)$.

When $N=p$ is prime, there are only two cusps on $X_0(p)$, usually denoted by $[0]$ and $[\infty]$. 
In \cite{Ogg1973}, Ogg computed that the divisor class of $[0]-[\infty]$ in $J_0(p)$ has order 
$$
n=(p-1)/\gcd(p-1, 12). 
$$ 
Both cusps are rational, so $\cC(p)=\cC_p(\Q)=\cC_p$. Based on numerical experimentations, Ogg conjectured in \cite{Ogg1975} that 
\begin{conj*}{CJ-p}\label{conjCJ-p} The cyclic group $\cC(p)$ is the full torsion subgroup of $J_0(p)(\Q)$. 
\end{conj*}

Mazur proved this conjecture in \cite{Mazur1977} using his theory of the Eisenstein ideal of the Hecke algebra of level $p$. There is a natural generalization of this conjecture to arbitrary $N$, nowadays called \textit{generalized Ogg conjecture}: 

\begin{conj*}{CJ-N}\label{conjCJ-N} For arbitrary $N\geq 1$, we have $\cC(N)=\cC_N(\Q)=J_0(N)(\Q)_\tor$.
\end{conj*}

Note that Conjecture \ref{conjCJ-N} is a combination of two conjectures. The first is the equality $\cC(N)=\cC_N(\Q)$, which concerns only the cuspidal divisor group; this was explicitly conjectured by Hwajong Yoo in \cite{YOO_rational_2023} in response to a question of Ribet. 
The second is the equality $\cC_N(\Q)=J_0(N)(\Q)_\tor$, which predicts that all $\Q$-rational torsion points on $J_0(N)$ arise from cuspidal divisors, so there are no ``unexpected" $\Q$-rational torsion points on $J_0(N)$. 

Determining the structure of the cuspidal divisor group for general $N$ is quite complicated. Recently, extending an earlier work of Lorenzini, Ling, Takagi, and others, the structure of $\cC(N)$ has been completely determined by Yoo \cite{YOO_rational_2023}. Yoo's proof combines a careful study of modular units with explicit construction of appropriate rational cuspidal divisors. Consequently, if Conjecture \ref{conjCJ-N} is true, then the structure of $J_0(N)(\Q)_\tor$ is determined as well. 

The conjecture $\cC(N)=\cC_N(\Q)$ is known in some cases, but not in general. 
Note that for square-free $N$, we trivially have $\cC(N)=\cC_N(\Q)=\cC_N$ because all the cusps are rational. The known non-trivial cases are essentially those $N$ whose prime decomposition contains at most two prime factors with higher powers;  the interested reader 
may consult the introduction of \cite{YOO_rational_2023} for a list of known cases. 

Given a finite abelian group $G$ and a prime $\ell$, we denote by $G_\ell$ the $\ell$-primary subgroup of~$G$. Note that the equality 
\begin{equation}\label{eqC=J}
\cC(N)_\ell = (J_0(N)(\Q)_\tor)_\ell 
\end{equation}
implies $\cC(N)_\ell=\cC_N(\Q)_\ell$. The strongest result towards the equality \eqref{eqC=J} is the following recent result of Yoo \cite{YOO_torsion_2023}: 

\begin{thm}\label{thmYoo} Let $N$ be a positive integer. Let $\ell$ be an odd prime whose square does not divide~$N$. If $\ell \geq 5$, then \eqref{eqC=J} holds. If $\ell=3$, then \eqref{eqC=J} holds 
under the additional assumption that either $3\nmid N$, or $N$ has a prime divisor congruent to $-1$ modulo $3$.
\end{thm}

Yoo's proof is based on the theory of Eisenstein ideals, with important  refinements introduced by Ohta in \cite{ohta_eisenstein_2014}, where Theorem \ref{thmYoo} is proved for square-free $N$. 

\begin{rem}
\begin{enumerate}
    \item When $N = p^r$ is a prime power, \eqref{eqC=J} holds for all odd $\ell\neq p$, so the technical conditions on $\ell=3$ can be dropped; cf. \cite[(1.1)]{YOO_torsion_2023}. Moreover, 
    in this case there is a different geometric approach to this problem due to Lorenzini. 
Let $\cJ_0(N)$ denote the N\'eron model of $J_0(N)$ over $\Spec(\Z)$. Let $\cJ_0(N)_{\F_p}$ denote the fibre of $\cJ_0(N)$ at $p$, and let $\cJ_0(N)_{\F_p}^0$ be its connected component of the identity. Let $\Phi_{N,p}=\cJ_0(N)_{\F_p}/\cJ_0(N)_{\F_p}^0$ be the group of connected components of $\cJ_0(N)$ at $p$. There is a canonical reduction map $J_0(N)(\Q)_\tor\to \cJ_0(N)_{\F_p}$, which is \textit{injective}; cf. Appendix in \cite{Katz}. Composing this reduction map with $\cJ_0(N)_{\F_p} \to \Phi_{N,p}$, one obtains a canonical homomorphism $\pi_p\colon J_0(N)(\Q)_\tor\to \Phi_{N,p}$. The idea of Lorenzini's approach \cite{Lorenzini1995} is that, for certain $\ell$, the map $\pi_p\colon (\cC(N))_\ell\to (\Phi_{N,p})_\ell$ is surjective, whereas $\pi_p\colon (J_0(N)(\Q)_\tor)_\ell\to (\Phi_{N,p})_\ell$ is injective. This is based on complicated computations with $(\Phi_{N,p})_\ell$ and some inductive arguments with $r$. We note that for prime $N=p$, Mazur already proved in \cite{Mazur1977} that $\pi_p$ is an isomorphism. 
    \item The proofs of special cases of Conjecture \ref{conjCJ-N} by Lorenzini, Mazur, Ohta, and Yoo implicitly rely on the knowledge of the order of $\cC(N)$. In \cite{RW}, Ribet and Wake prove the equality \eqref{eqC=J} for square-free $N$ and $\ell\nmid 6N$ 
    by showing that both sides are isomorphic modules for the Hecke algebra acting on $J_0(N)$, without actually computing either side explicitly. 
    \item The methods that have been used to prove \eqref{eqC=J} in certain cases so far fall short of proving this equality for $\ell=2$. Thus, \eqref{eqC=J} remains largely open for the $2$-primary part, except when $N=p$ is prime. We note that the most technical arguments in Mazur's paper \cite{Mazur1977} proving Conjecture \ref{conjCJ-p} are actually concerned with the $2$-primary torsion.  
    \item In \cite{OhtaJ1}, Ohta proves that $(J_1(p)(\Q)_\tor)_\ell$ is generated by cuspidal divisors for any odd prime $\ell$, as was conjectured by Conrad, Edixhoven and Stein \cite{CES}. The proof again relies on the Eisenstein ideal machinery. 
\end{enumerate}
\end{rem}

Conjecture \ref{conjCJ-p} has a ``dual" version concerning the largest $\mu$-type subgroup $\cM(N)$ of $J_0(N)$: a commutative group scheme is \textit{$\mu$-type} if it is finite, flat, and its Cartier dual is a constant group scheme. The natural morphism $X_1(N)\to X_0(N)$ induces, by Picard functoriality, a morphism of their Jacobians $J_0(N)\to J_1(N)$. The kernel of this last morphism is called the \textit{Shimura subgroup} of $J_0(N)$, and will be denoted by $\cS(N)$. 
When $N=p$ is prime, $\cS(p)$ is cyclic of the same order as  $\cC(p)$. In \cite{Ogg1975}, Ogg conjectured:
\begin{conj*}{SJ-p}\label{conjSJ-p} For any prime $p$, we have $\cS(p) = \cM(p)$.
\end{conj*}

This conjecture was proved by Mazur in \cite{Mazur1977}. Despite the dual appearance of 
Conjectures \ref{conjCJ-p} and \ref{conjSJ-p}, the proof of Conjecture \ref{conjSJ-p} lies 
deeper than the proof of Conjecture \ref{conjCJ-p} since its proof relies on the fact that the completion of the Hecke algebra at any prime ideal in the support of Eisenstein ideal is Gorenstein. For general $N$, the order of $\cS(N)$, the action of the Hecke operators and the action of $\Gal(\overline{\Q}/\Q)$ on $\cS(N)$ are computed in \cite{LO}. In particular, $\cS(N)\subseteq \cM(N)$. The following generalization of Mazur's result was proved by Vatsal \cite{Vatsal}:  

\begin{thm}\label{Vatsal} Let $\ell$ be an odd prime whose square does not divide $N$.
Then $\cS(N)_\ell=\cM(N)_\ell$. 
\end{thm}

Vatsal's proof is very different from Mazur's. It combines a theorem of  
Ihara about unramified coverings of modular curves over finite fields with some deep results 
about the supersingular reductions of CM points on $X_0(N)$ and special values of $L$-functions. 

\begin{rem}
    By definition, $\cS(N)$ is the kernel of a natural homomorphism $J_0(N)\to J_1(N)$. One might wonder whether the quotients of $J_0(N)$ by subgroups of $\cC(N)$ also give arithmetically interesting abelian varieties. Incidentally, there is another conjecture of Ogg in that direction. Assume $N$ is square-free and divisible by an even number of primes. Let $\cO_N$ be a maximal order in the indefinite quaternion algebra over $\Q$ with discriminant $N$, and let $\cO_N^1$ be the subgroup of $\cO_N$ consisting of elements of reduced norm $1$. Then $\cO_N^1$ is isomorphic to a subgroup of $\SL_2(\R)$, so it acts on the upper half-plane $\cH$. The quotient $X^N\colonequals \cO_N^1\bs \cH$ is a Shimura curve and has a canonical model over $\Q$. In \cite{Ribet1980}, Ribet proved the existence of an isogeny defined over $\Q$ between the ``new" part $J_0(N)^\new$ of $J_0(N)$ and the Jacobian $J^N$ of $X^N$. Unfortunately, Ribet's proof provides no information about this isogeny beyond its existence. 
    In \cite{Ogg1985}, by studying the component groups of $J_0(N)$ and $J^N$, Ogg made an explicit conjecture about the kernel of Ribet’s isogeny when $N=pq$ is a product of two distinct primes and $p = 2,3,5,7,13$: the conjecture predicts that there is an isogeny $J_0(N)^\new\to J^N$ of minimal degree whose kernel is a specific subgroup of $\cC_N$. Although this conjecture is false in general \cite{KP2}, it seems to be correct when $J_0(N)^\new=J_0(N)$; cf. \cite{KP1}. 
\end{rem}


\section{Review of Drinfeld modules}\label{sDM}

Drinfeld modules can be defined for subrings of function fields 
of general curves over a finite field $\F_q$ with $q$ elements. 
On the other hand, the analogue of $\Q$ in this setting is the rational function field 
$F=\F_q(T)$, where $T$ is an indeterminate. Since Ogg's conjectures are explicit classifications 
of rational torsion points of elliptic curves over $\Q$ or Jacobians of modular curves, their 
analogues over function fields are also stated mostly over $F$. Thus, 
in this section, we discuss Drinfeld modules only for the polynomial ring  
$A=\F_q[T]$. This assumption also simplifies some of the non-essential technicalities 
of the theory. Most of the results in this section are due to Drinfeld \cite{Drinfeld}. 
Detailed proofs of these statements can be found in \cite{Goss} or \cite{PapikianGTM}. 


\subsection{Analytic theory}\label{ssDM} 
We start with the analytic theory of Drinfeld modules, which parallels closely the theory of roots of unity $e^{2\pi i/n}$ and also analytic uniformization of elliptic curves as quotients of $\C$ by lattices. 

First, we introduce the analogues 
of real and complex numbers in the function field setting. 
The degree function $\deg=\deg_T\colon A\to \Z_{\geq 0}\cup \{-\infty\}$, which assigns to $0\neq a\in A$ its degree 
as a polynomial in $T$ and $\deg_T(0)=-\infty$, extends to $F$ by $\deg(a/b)=\deg(a)-\deg(b)$. 
The map $-\deg$ is a valuation on $F$; the corresponding place of $F$ is usually denoted by $\infty$. 
Let $\abs{\cdot}$ denote the corresponding absolute value on $F$ normalized by $\abs{T}=q$. 
The completion $\Fi$ of $F$ with respect to this absolute value is isomorphic to the field $\ls{\F_q}{1/T}$ 
of Laurent series in $1/T$. Finally, let $\Ci$ be the completion of an algebraic closure of $\Fi$. The absolute value 
$\abs{\cdot}$ has a unique extension, also denoted by $\abs{\cdot}$, to $\Ci$. 

An \textit{$A$-lattice} $\La\subset \Ci$ of rank $r\geq 1$ is a discrete 
$A$-submodule of $\Ci$ of rank $r$, where ``discrete" means that for any $N>0$ the set $\{\la\in \La\ :\ \abs{\la}\leq N\}$ is finite. One shows that any $A$-lattice
is of the form $\La=A\omega_1+\cdots +A\omega_r$, where $\omega_1, \dots, \omega_r\in \Ci$ are linearly independent over $\Fi$. 
Since the degree of $\Ci$ over $\Fi$ is infinite, there are $A$-lattices of arbitrarily large ranks (unlike $\Z$-lattices in $\C$). 

The \textit{exponential function} of $\La$ is 
$$
e_\La(x)=x\prod_{0\neq \la\in \La} \left(1-\frac{x}{\la}\right). 
$$
Using the discreteness of $\La$, it is not hard to show that the function $e_\La(x)$ is entire, i.e., converges everywhere on $\Ci$. Because of the non-archimedean setting, this implies that $e_\La(x)\colon \Ci\to \Ci$ is surjective. 
The set of zeros of $e_\La$ is exactly $\La$. 
Finally, because $\La$ is an $\F_q$-vector space,  $e_\La(x)$ satisfies $e_\La(x+y)=e_\La(x)+e_\La(y)$ and $e_\La(\alpha x)=\alpha e_\La(x)$ for all 
$\alpha\in \F_q$. In other words, $e_\La(x)$ is an $\F_q$-linear function. Thus, the power series expansion of $e_\La(x)$  is of the form 
$$
e_\La(x)=\sum_{n\geq 0} e_n(\La)x^{q^n}. 
$$
There are recursive formulas for the coefficients of $e_\La$ in terms of Eisenstein series:
If we put 
\begin{equation}\label{eqEisSer}
    E_n(\La)=\sum_{0\neq \la\in \La} \frac{1}{\la^n},
\end{equation}
then 
\begin{equation}\label{eqCoefExp}
    e_n(\La) = E_{q^n-1}(\La)+\sum_{i=1}^{n-1}e_i(\La) E_{q^{n-i}-1}(\La)^{q^i}. 
\end{equation}
Let 
$$\atwist{\Ci}{x}=\{a_0x+a_1x^q+\cdots+a_n x^{q^n}\mid n\geq 0, a_0, \dots, a_n\in \Ci\}$$ be the non-commutative ring 
of $\F_q$-linear polynomials with usual addition of polynomials but where multiplication is defined via the composition of polynomials. For example, 
$$
(Tx+x^q)\circ (x+Tx^{q^2})=T(x+Tx^{q^2})+(x+Tx^{q^2})^q= Tx+x^q+T^2x^{q^2}+T^qx^{q^3}. 
$$
Given $f(x)=a_0x+a_1x^q+\cdots+a_n x^{q^n}$ in $\atwist{\Ci}{x}$, 
we denote $\partial f=\frac{\dd}{\dd x}f(x)=a_0$. 

An important property of $e_\La(x)$ is the functional equation 
$$
e_\La(T x) = \phi_T^\La(e_\La(x)),
$$
where $\phi_T^\La(x)=Tx+g_1(\La)x^q+\cdots+ g_r(\La)x^{q^r}\in \atwist{\Ci}{x}$, $g_r(\La)\neq 0$. Because $e_\La(x)$ is $\F_q$-linear, this 
functional equation extends to all $a\in A$: for each 
$a\in A$, there is $\phi_a^\La(x)\in \atwist{\Ci}{x}$ such that $\deg_x \phi_a^\La(x)=\abs{a}^r$, $\partial \phi_a^\La=a$, and 
$
e_\La(a x)= \phi_a^\La(e_\La(x))$. 
Moreover, the map 
\begin{align*} 
	\phi^\La\colon A &\To \atwist{\Ci}{x}, \\ a &\longmapsto \phi_a^\La(x) 
	\end{align*} 
	is an $\F_q$-algebra homomorphism, called the \textit{Drinfeld module 
of rank $r$ associated to $\La$}. 

Conversely, a \textit{Drinfeld $A$-module of rank $r$ over $\Ci$} is an $\F_q$-algebra homomorphism $\phi\colon A\to \atwist{\Ci}{x}$, $a\mapsto \phi_a(x)$,  defined by 
$\phi_T(x)=Tx+g_1x^q+\cdots +g_r x^{q^r}$ with $g_r\neq 0$. One constructs an entire $\F_q$-linear function $e_\phi(x)$ 
satisfying $e_\phi(Tx)=\phi_T(e_\phi(x))$ as follows. Put 
$$
e_\phi(x) = e_0x+e_1x^q+e_2x^{q^2}+\cdots,
$$ 
where $e_0, e_1, \dots$ are to be determined. The functional equation $e_\phi(Tx)=\phi_T(e_\phi(x))$ leads to a system 
of equations 
$$
(T^{q^n}-T)e_n=e_{n-1}^qg_1+e_{n-2}^{q^2}g_2+\cdots+ e_{n-r}^{q^r}g_r, \quad n\geq 0,
$$
where $e_i=0$ for $i<0$. If we put $e_0=1$, then every other $e_n$ is uniquely determined from the above recursive formulas. 
It is not hard to show that the resulting function $e_\phi(x)$ is entire and the set of zeros $\La_\phi$ of $e_\phi$ is an 
$A$-lattice of rank $r$ such that $\phi=\phi^{\La_\phi}$.  

A \textit{morphism} $u\colon \phi\to \psi$ of Drinfeld modules is a polynomial $u(x)\in \atwist{\Ci}{x}$ such that $u(\phi_a(x))=\psi_a(u(x))$ 
for all $a\in A$. A morphism $u\colon \phi\to \psi$ is an \textit{isomorphism} 
if $u$ is invertible, i.e., $u=c\in \Ci^\times$. 
Note that since $T$ generates the $\F_q$-algebra $A$, the commutation 
\begin{equation}\label{eq-defmor}
    u(\phi_T(x))=\psi_T(u(x))
\end{equation}
is sufficient to ensure $u(\phi_a(x))=\psi_a(u(x))$  for all $a\in A$. 
Comparing the degrees of both sides of \eqref{eq-defmor}, we see that nonzero morphisms can exist only between Drinfeld modules of the same rank. 
We denote the set of all morphisms $\phi\to \psi$ by $\Hom(\phi, \psi)$. 
A \textit{morphism} of lattices $\La\to \La'$ is an element $c\in \Ci$ 
such that $c\La\subseteq \La'$. The set 
of all morphisms $\La\to \La'$ is denoted $\Hom(\La, \La')$. Both $\Hom(\phi, \psi)$ and $\Hom(\La, \La')$ are naturally $A$-modules. One shows that there is an isomorphism of $A$-modules 
\begin{align*}
\Hom(\phi, \psi) &\overset{\sim}{\To} \Hom(\La_\phi, \La_\psi)\\ 
u &\longmapsto \partial u. 
\end{align*}
From these constructions we get:  
\begin{thm} 
The category of Drinfeld modules of rank $r$ over $\Ci$ and 
the category of $A$-lattices of rank $r$ in $\Ci$ are equivalent. 
\end{thm}

\begin{example} The \textit{Carlitz module} is the Drinfeld module defined by $\phi_T=Tx+x^q$. This is the simplest possible 
	Drinfeld module. We will distinguish the Carlitz module among all other Drinfeld modules by denoting it by $C$, i.e., $C_T=Tx+x^q$. 
	The rank of $C$ is $1$. It is easy to show that the coefficients of $e_C(x)=\sum_{n\geq 0} e_n x^{q^n}$ are given by the formula 
	$$
e_0=1 \quad \text{and}\quad  e_n=(T^{q^n}-T)(T^{q^{n}}-T^q)\cdots (T^{q^n}-T^{q^{n-1}})\quad \text{for}\quad n\geq 1. 
$$
The lattice $\La_C$ of $C$ has rank $1$, so $\La=\pi_C A$ for some $\pi_C\in \Ci^\times$. The generator $\pi_C$, which is well-defined 
up to an $\F_q^\times$-multiple, is called the \textit{Carlitz period}; it plays the role of $2\pi i\in \C$. 
There are various explicit formulas for the Carlitz period, one of which is the following:
\begin{equation}\label{eqCarlitzFormua}
	\sum_{\substack{a\in A \\ a\text{ monic}}} \frac{1}{a^{q-1}}=-\frac{\pi_C^{q-1}}{T^{q}-T}.
\end{equation}
This is the analogue of Euler's formula $\sum_{n\geq 1} 1/n^2 =\pi^2/6$. Wade \cite{Wade} proved that $\pi_C$ is transcendental over $F$, just like $\pi$ is transcendental over $\Q$. 
\end{example}

\begin{rem}
	The Carlitz module was originally introduced by Carlitz in \cite{Carlitz}, where he 
 defined $e_C(x)$ and proved \eqref{eqCarlitzFormua}. 
Carlitz also showed  in \cite{CarlitzTAMS} that $C$ gives rise to the correct analogue of cyclotomic polynomials over $F$. 
Unfortunately, \cite{Carlitz}  and \cite{CarlitzTAMS} did not receive the attention they deserved from the larger mathematical community and were mostly forgotten until the 1970s. 
\end{rem}

Given a Drinfeld module $\phi$, we equip $\Ci$ with a new $A$-module structure $a\circ z =\phi_a(z)$ denoted ${^\phi}\Ci$. 
This gives an exact sequence of $A$-modules 
\begin{equation}\label{eqDMUnif}
	0\To \La_\phi \To \Ci \overset{e_\phi}{\To} {^{\phi}}\Ci\To 0, 
\end{equation}
which can be interpreted as the analogue of analytic uniformization $\C/\La\overset{\sim}{\To}E(\C)$, 
$z\mapsto (\wp(z), \wp'(z))$ of an elliptic curve $E$ over $\C$ (here $\wp$ is the Weierstrass $\wp$-function associated to the lattice $\La$). 

For nonzero $a\in A$, the \textit{$a$-torsion points} of $\phi$ are the roots of $\phi_a(x)$. The set of these roots, denoted $\phi[a]$, is naturally an $A$-module 
via $b\circ \alpha = \phi_b(\alpha)$, where $b\in A$ and $\alpha\in \phi[a]$ (to see that $b\circ \alpha$ is in $\phi[a]$, compute 
$\phi_a(b\circ \alpha)=\phi_a(\phi_b(\alpha))=\phi_b(\phi_a(\alpha))=\phi_b(0)=0$). Applying the snake lemma to the exact sequence \eqref{eqDMUnif}, 
we get an isomorphism of $A$-modules: 
$$
\phi[a]\cong \La_\phi/a\La_\phi\cong (A/aA)^r. 
$$
A morphism $u\colon \phi\to \psi$ induces a homomorphism ${^\phi}\Ci\to {^\psi}\Ci$ of $A$-modules. 
In particular, $u$ induces a homomorphism $\phi[a]\to \psi[a]$, $\alpha\mapsto u(\alpha)$. 

It is easy to show that two Drinfeld modules $\phi$ and $\psi$ are isomorphic if and only if their corresponding lattices are homothetic: $\La_\phi= c\La_\psi$ for $c\in\Ci^\times$. Thus, to classify Drinfeld modules of rank~$r$ up to isomorphism, it is enough to classify $A$-lattices in $\Ci$ of rank $r$ up to homothety. 
This is trivial if $r=1$. For $r\geq 2$ the key object for our task 
is the \textit{Drinfeld symmetric space}
 $$\Omega^r= \p^{r-1}(\Ci)-\bigcup_{\Fi\text{-rational H}} H,
 $$ 
 where the union is over the $\Fi$-rational hyperplanes; equivalently, $\Omega^r$ is the set of all points $(z_1, \dots, z_r)$ of $\p^{r-1}(\Ci)$
 such that $z_1, \dots, z_r$ are linearly independent over $\Fi$. To the point $(z_1, \dots, z_r)$ 
 we associate the homothety class of the lattice $Az_1+\cdots+Az_r$. 
 The action of $\GL_r(\Fi)$ on $\p^{r-1}(\Ci)$ 
 preserves $\Omega^r$. 

 \begin{example}
 	Suppose $r=2$. In this case, $\p^{r-1}(\Ci)=\p^1(\Ci)$ consists of $[1, 0]$ and $[z, 1]$, $z\in \Ci$. It is easy to see that 
 	\begin{align*}\Omega^2 &=\{[z, 1]\mid z\not\in \Fi\}\\ &=\Ci-\Fi.
 	\end{align*}
 	After identifying $\Omega^2$ with $\Ci-\Fi$, $\GL_2(\Fi)$ acts on $\Omega^2$ by linear fractional transformations 
 	$$
 	\begin{pmatrix} a & b \\ c & d\end{pmatrix} z = \frac{az+b}{cz+d}. 
 	$$
 \end{example}
 
 From the bijection between lattices and Drinfeld modules, one deduces that the set of 
 isomorphism classes of Drinfeld modules of rank $r$ over $\Ci$ is in natural bijection with the set of orbits $$\GL_r(A)\bs\Omega^r.$$ 

 Each nonzero ideal $\fn\lhd A$ has a unique monic generator, which, by abuse of notation, we will also denote by $\fn$. Define the subgroups $\G(\fn)\subseteq \G_1(\fn)\subseteq \G_0(\fn)\subseteq \GL_r(A)$ as 
 $$
 \G(\fn)=\left\{\gamma\in \GL_r(A)\mid \gamma\equiv 1 \Mod{\fn}\right\},
 $$
 $$
 \G_1(\fn)= \left\{\begin{pmatrix} c_{ij}\end{pmatrix}\in \GL_r(A)\mid \begin{pmatrix} c_{11}\\ c_{21}\\ \vdots\\ c_{r1} \end{pmatrix}\equiv \begin{pmatrix} 1\\ 0\\ \vdots\\ 0 \end{pmatrix} \Mod{\fn}\right\}, 
 $$
 $$
 \G_0(\fn)= \left\{\begin{pmatrix} c_{ij}\end{pmatrix}\in \GL_r(A)\mid \begin{pmatrix} c_{11}\\ c_{21}\\ \vdots\\ c_{r1} \end{pmatrix}\equiv \begin{pmatrix} \ast\\ 0\\ \vdots\\ 0 \end{pmatrix} \Mod{\fn}\right\}. 
 $$
 Extending the above bijection, one shows that the orbits $Y^r_0(\fn)(\Ci)\colonequals \G_0(\fn)\bs \Omega^r$ are in bijection with the isomorphism classes 
 of pairs $(\phi, G_\fn)$, where $\phi$ is a Drinfeld module of rank~$r$ and $G_\fn\subset \phi[\fn]$ is an $A$-submodule isomorphic to $A/\fn$ 
 (two pairs $(\phi, G_\fn)$ and $(\phi', G_\fn')$ are isomorphic if there is an isomorphism $u\colon \phi\overset{\sim}{\to}\phi'$ such that $u(G_n)=G_n'$). 
 Similarly, $$Y^r_1(\fn)(\Ci):=\G_1(\fn)\bs \Omega^r \longleftrightarrow \{\text{isomorphism classes 
 of pairs } (\phi, P_\fn)\},$$ where $\phi$ is a Drinfeld module of rank $r$ and $P_\fn\in \phi[\fn]$ is a torsion point which generates an $A$-submodule isomorphic to $A/\fn$. Finally, 
 $Y^r(\fn)(\Ci):=\G(\fn)\bs \Omega^r$ classifies the isomorphism classes of pairs $(\phi, \iota)$, where 
 $\iota$ is an isomorphism $(A/\fn)^r\overset{\sim}{\To}\phi[\fn]$. 
 
 The quotients $Y=Y^r(\fn), Y^r_1(\fn), Y^r_0(\fn)$ are much more than just sets. In \cite{Drinfeld}, Drinfeld showed that 
 $\Omega^r$ has a natural structure of a smooth rigid-analytic manifold over $\Fi$, and that the group $\GL_r(A)$ acts discontinously on $\Omega^r$, so 
 $Y$ has a structure of an analytic manifold over $\Fi$. 
 Moreover, he proved that  $Y$ is algebraizable, in the sense 
 that it is the analytic space corresponding to an affine algebraic variety over $\Fi$: this point of view will be discussed at the end of Section~\ref{subs-DMalgtheory}. For more about analytic geometry over non-archimedean fields, the reader might consult \cite{FvdP}. 


\subsection{Algebraic theory}\label{subs-DMalgtheory} An \textit{$A$-field} is a field $K$ equipped with an $\F_q$-algebra homomorphism $\gamma\colon A\to K$. We will denote $t=\gamma(T)$. 
We call a nonzero prime 
ideal $\fp\lhd A$ a \textit{prime} of $A$; the primes of $A$ are in bijection with monic irreducible polynomials of $A$ of positive degree. 
Since $K$ is a field, there are two 
possibilities for the kernel of $\gamma$: either $\ker(\gamma)=0$ or $\ker(\gamma)=\fp$ is a prime of $A$. We call $\chr_A(K)\colonequals \ker(\gamma)$ 
the \textit{$A$-characteristic} of $K$. 

Let $\atwist{K}{x}$ be the non-commutative ring of $\F_q$-linear polynomials with coefficients in $K$, 
defined as earlier for $K=\Ci$. The ring $\atwist{K}{x}$ 
can be identified with the ring of $\F_q$-linear endomorphisms of the additive group scheme $\gm_{a, K}$  over $K$. 
A \textit{Drinfeld module} of rank $r\geq 1$ over $K$ is an $\F_q$-algebra homomorphism $\phi\colon A \to \atwist{K}{x}$, $a\mapsto \phi_a(x)$, such that 
$$
\phi_T(x)=tx+g_1x^q+\cdots+g_rx^{q^r}, \quad g_r\neq 0. 
$$
Note that the homomorphism $\phi$ is always injective, even if $\chr_A(K)\neq 0$, so 
$\phi$ gives an embedding of the commutative ring $A$ into the non-commutative ring $\atwist{K}{x}$. From the definition it follows 
that $\partial \phi_a=\gamma(a)$ for all $a\in A$. Hence, $\phi$ equips $\gm_{a, K}$ with a new action of $A$ such that on the tangent space of $\gm_{a, K}$ around the origin the induced action of $A$ is  
via the structure morphism $\gamma$.

Given a polynomial 
$f(x)=a_0x+a_1x^q +\cdots+ a_dx^{q^d}\in \atwist{K}{x}$, the smallest index $0\leq h\leq d$ 
such that $a_h\neq 0$ is the \textit{height} of $f$, denoted  $\Ht(f)$. 
If $\chr_A(K)=\fp\neq 0$, then $\gamma$ factors through the quotient $A\to A/\fp$ 
and $\Ht(\phi_\fp)>0$.  Using the commutation $\phi_a(\phi_\fp(x))=\phi_\fp(\phi_a(x))$, 
it is not hard to show that there is an integer $1\leq H(\phi)\leq r$, 
called the \textit{height of $\phi$}, such that for all nonzero $a\in A$ we have 
$$
\Ht(\phi_a)=H(\phi)\cdot \ord_\fp(a)\cdot \deg(\fp),
$$
where $\ord_\fp(a)$ is the power with which $\fp$ divides $a$. 

Given a Drinfeld module $\phi$ over $K$ of rank $r$ and $0\neq a\in A$, let $\phi[a]$ be the set of roots of $\phi_a(x)$ in an algebraic 
closure $\oK$ of $K$ without repetitions. As over $\Ci$, the set $\phi[a]$ is naturally an $A$-module. Decomposing $a=\fp_1^{s_1}\cdots\fp_m^{s_m}$ 
into distinct prime powers, we get an isomorphism of $A$-modules
$$
\phi[a]\cong \phi[\fp_1^{s_1}]\times \cdots \times  \phi[\fp_m^{s_m}].
$$
Moreover, for a prime $\fp$ we have 
$$
\phi[\fp^n]\cong
\begin{cases}
	(A/\fp^n)^r, & \text{if } \fp\neq \chr_A(K); \\ 
	(A/\fp^n)^{r-H(\phi)}, & \text{if } \fp = \chr_A(K). 
\end{cases}
$$

Assume that $\chr_A(K)\nmid a$.  In this case the polynomial $\phi_a(x)\in K[x]$ is separable, so 
the $A$-module $\phi[a]$ is naturally equipped with an action of the absolute Galois group $G_K\colonequals\Gal(K^\sep/K)$ of $K$. 
Since this action commutes with the action of $A$, we obtain a representation 
\begin{equation}\label{eqResRep}
    \rho_{\phi, a}: G_K\To \Aut_A(\phi[a])\cong \GL_r(A/aA). 
\end{equation}

\begin{example}
	Let $K=F$ and $\gamma\colon A\to F$ be the natural embedding of $A$ into its field of fractions. Consider the Carlitz module defined by 
	$C_T=Tx+x^q$ as a Drinfeld module of rank $1$ over $F$. The extensions $F(C[a])/F$ are the analogues of cyclotomic extensions of~$\Q$; cf. \cite{Hayes}.  
	For example, $F(C[a])/F$ is unramified at the primes of $A$ not dividing $a$ and $\Gal(F(C[a])/F) \cong (A/aA)^\times$, 
	where the isomorphism is given by mapping the Frobenius at $\fp$ to the residue of $\fp$ modulo $a$. 
\end{example}

Repeating an earlier definition, a \textit{morphism} 
$u\colon \phi\to \psi$ is an element $u\in \atwist{K}{x}$ such that $u\circ \phi_T=\psi_T\circ u$ (and thus, $u\circ \phi_a=\psi_a\circ u$ 
for all $a\in A$). The set $\Hom_K(\phi, \psi)$ of all morphisms $\phi\to \psi$ is naturally an $A$-module, with 
the action of $a\in A$ on $u\in \Hom_K(\phi, \psi)$ defined by $a * u=u\circ \phi_a=\psi_a\circ u$.  
A basic fact of the theory of Drinfeld modules is that 
if $\phi$ and $\psi$ have rank $r$, then $\Hom_K(\phi, \psi)$ is a free $A$-module of rank $\leq r^2$. 
Denote $\End_K(\phi):=\Hom_K(\phi, \phi)$; this 
is the centralizer of $\phi(A)$ in $\atwist{K}{x}$. 

\begin{example}
	Let $K= \F_{q^2}$ and $\gamma\colon A\to A/(T)\cong \F_q \hookrightarrow K$.  
	Let $\phi\colon A\to \twist{K}$ be the Drinfeld module of rank $2$ defined by $\phi_T=tx+x^{q^2}=x^{q^2}$. 
	In this case we have 
	$\End_K(\phi)=\atwist{K}{x}$, since $\phi_T$ is in the center of $\atwist{K}{x}$. Now it is not hard to see 
	that $\End_K(\phi)$ has rank~$4$ over $A$. But one can say more: $\End_K(\phi)$ is a maximal $A$-order 
	in the quaternion division algebra over $F$ ramified at $T$ and $\infty$. To see this, assume for simplicity that $q$ is odd. 
	Fix a non-square $\alpha$
	in $\F_q^\times$ and let $j\in \F_{q^2}$ be such that $j^2=\alpha$. The conjugate of $j$ over $\F_q$ is $-j=j^q$. 
	Thus, $\F_{q^2}\cong \F_q(j)$, and $x^q\circ  j=j^q\circ x^q=-j\circ x^q$. 
	If we denote $i=x^q$, then $i\circ i=\phi_T$, and we see that $\twist{K}\cong A[i, j]$, where $i^2=T$, $j^2=\alpha$, $ij=-ji$. 
\end{example}

\begin{example} If $K=\Ci$, then 
	$$
	\End_{\Ci}(\phi) \overset{\sim}{\To}\{c\in \Ci\mid c\La_\phi\subseteq \La_\phi\}. 
	$$
	This can be used to show that $\End_{\Ci}(\phi)$ is commutative and the rank of $\End_{\Ci}(\phi)$  
	as an $A$-module divides $r$. To construct a Drinfeld module with large endomorphism ring in this setting, 
	let $L/F$ be a field extension such that there is a unique place in $L$ over $\infty$. 
	Let $B$ be the integral closure of $A$ in $L$. Let $1=z_1, z_2, \dots, z_r\in B$ be a basis of $L$ as a vector space over $F$. 
	Put 
	$$
	\La=Az_1+Az_2+\cdots +Az_r. 
	$$
	Then $\La\subset \Ci$ is an $A$-lattice of rank $r$ and $\End_{\Ci}(\phi^\La)\cong B$. 
\end{example}

The action of $u\in \End_K(\phi)$ on $\phi[a]$ commutes with the action of $G_K$. If $\fp\neq \chr_A(\phi)$, then there is an injective homomorphism 
$$
\End_K(\phi)\otimes_A A/aA\To \End_{A[G_K]}(\phi[a]). 
$$
Thus, the image of $G_K$ in $\Aut(\phi[a])\cong \GL_r(A/aA)$ is proportionally smaller to the size of $\End_K(\phi)$ (larger the endomorphism ring, smaller the image of $G_K$ will be). 

\begin{defn}
A morphism $u\colon \phi\to \psi$ over $K$ is an \textit{isomorphism} if it has an inverse in $\atwist{K}{x}$. Hence, an isomorphism is given by a 
nonzero constant $c\in K$ such that $\phi_T(x)=c\psi_T(c^{-1}x)$. If $\phi_T=tx+g_1x^q+\cdots+g_rx^{q^r}$ and 
$\psi_T=tx+h_1x^q+\cdots+h_rx^{q^r}$, then this is equivalent to $g_i=h_i c^{q^i-1}$ for all $1\leq i\leq r$. 
\end{defn}

\begin{example}
		Suppose $\phi$ and $\psi$ have rank $1$. 
	Let $c$ be a root of $x^{q-1}=g_1/h_1$. Then $c\phi_Tc^{-1}=\psi_T$, so $\phi$ and $\psi$ are isomorphic 
	over $K(\sqrt[q-1]{g_1/h_1})$. This implies that, up to isomorphism, the Carlitz module $C_T=tx+x^q$ 
	is the only Drinfeld module of rank $1$ over $K^\sep$. 
	But note that $x^{q-1}=g_1/h_1$ might not have roots in $K$, so $\phi$ and $\psi$ might not be isomorphic over $K$.
	
	Now suppose $\phi$ and $\psi$ have rank $2$. The \textit{$j$-invariant} of $\phi$ is 
	$$
	j(\phi):=g_1^{q+1}/g_2. 
	$$
	It is not hard to check that $\phi$ and $\psi$ are isomorphic over $\oK$ if and only if $j(\phi)=j(\psi)$. 
	
	For a Drinfeld module $\phi$ of rank $\geq 3$ there is a finite collection of $j$-invariants, which are rational functions in the 
	coefficients of $\phi$ and which distinguish the isomorphism class of $\phi$ over $\oK$. These were constructed by Potemine \cite{Potemine}. 
	These $j$-invariants are not algebraically independent, which is reflected in the fact that the moduli space of Drinfeld modules of rank $r\geq 3$ 
	is not isomorphic to the affine space $\A^{r-1}$.  
\end{example}

A \textit{$\G(\fn)$-structure} on a Drinfeld module $\phi$ of rank $r$ over $K$, if $\chr_A(K)\nmid \fn$, is an isomorphism of $A$-modules 
$$
\iota\colon (A/\fn)^r \To \phi[\fn](K),
$$
where $ \phi[\fn](K)$ is the set of $K$-rational $\fn$-torsion points of $\phi$. A \textit{$\G_1(\fn)$-structure} is an injective homomorphism 
$A/\fn \to \phi[\fn](K)$, and a \textit{$\G_0(\fn)$-structure} is a morphism $\phi\to \psi$ over $K$ whose kernel is isomorphic to $A/\fn$. 
The notion of Drinfeld module and level structures can be extended 
to an arbitrary $A$-scheme $S$ (instead of $S=\Spec(K)$): a Drinfeld module over $S$ is a homomorphism 
$\phi\colon A\to \End(\cL)$ from $A$ to the ring of $S$-endomorphisms of the group scheme underlying a line bundle $\cL$ on $S$ 
satisfying certain conditions, and $\G(\fn)$-structure is a homomorphism of $A$-modules $\iota: (A/\fn)^r \To \phi[\fn](S)$, which 
induces an equality of Cartier divisors $\phi[\fn](S) = \sum_{\alpha\in (A/\fn)^r} \iota(\alpha)$. There results a moduli functor 
$\cY^r(\fn)$ from the category of $A$-schemes to the category of sets which to an $A$-scheme $S$ associates the 
set $\cY^r(\fn)(S)$ of isomorphism classes of Drinfeld modules over $S$ of rank $r$ equipped with a $\G(\fn)$-structure. 
Under the assumption that $\fn$ is divisible by at least two primes, Drinfeld proved in \cite{Drinfeld} that  $\cY^r(\fn)$ 
is representable by an affine flat $A$-scheme $Y^r(\fn)$ of dimension $r-1$ whose fibres over $\Spec(A)$ are smooth 
away from $\fn$. Taking the quotients of $Y^r(\fn)$ by finite groups one constructs (coarse) moduli schemes $Y^r_1(\fn)$ and $Y^r_0(\fn)$ 
classifying Drinfeld modules with level $\G_1(\fn)$ and $\G_0(\fn)$-structures. 

\smallskip
Assume $r=2$. Let $Y_0(\fn)\colonequals Y_0^2(\fn)$ and denote by $X_0(\fn)$ the unique projective smooth geometrically connected curve containing $Y_0(\fn)$ as an affine subvariety. The set of points $X_0(\fn)-Y_0(\fn)$ are the cusps of $X_0(\fn)$.

\begin{example}
Over $\Ci$ the $j$-invariant of Drinfeld modules gives an isomorphism between the projective line $\p^1_{\Ci}$ and $X_0(1)$.  
	Let 
	\begin{equation}\label{eqUDM}
		\phi_T(x)=Tx+x^q+j^{-1}x^{q^2},
	\end{equation}
	be the ``universal" Drinfeld module of rank $2$ with $j$-invariant $j$, where we consider $j$ as a variable.  
	A cyclic $T$-submodule $\cC$ of $\phi$ is the set of roots of an $\F_q$-linear polynomial of the form $f_\cC(x)=x+\alpha x^q$, where 
	$\alpha$ is another variable 
	(we may assume that the coefficient of $x$ in $f_\cC(x)$ is $1$ because the polynomial is separable and the set of zeros does not change 
	if we multiply $f_\cC$ by a nonzero constant). Since $\cC\subset \phi[T]$, we have 
	$$
	\phi_T(x)= (Tx+\tilde{\alpha}x^q)\circ (x+\alpha x^q). 
	$$
	Thus, $j^{-1}=\tilde{\alpha}\alpha^q$ and $\alpha T+\tilde{\alpha}=1$. This leads to $j^{-1}=(1-\alpha T)\alpha^q$. Substituting $\alpha\mapsto -1/\alpha$, we obtain 
	$$
	T+\alpha+j^{-1}\alpha^{q+1} = 0. 
	$$
	Note that $\cC$ is automatically a $\phi(A)$-module, i.e., $f_\cC\circ \phi_T=g\circ f_\cC$ for some $g\in \atwist{\Ci}{x}$, since any root of $f_\cC(x)$ 
	maps to $0$ under the action of $\phi_T$. Thus, $X_0(T)$ is defined by 
	$$
	j=-\frac{\alpha^{q+1}}{\alpha+T}. 
	$$ 
	Similarly, $X_0(T+1)$ is defined by 
	$j=-\beta^{q+1}/(\beta+(T+1))$ for another variable $\beta$. A cyclic $T(T+1)$-submodule of $\phi$ decomposes into a direct product of cyclic $T$ and $T+1$ submodules. Changing the notation for the  variables  $\alpha$ and $\beta$ to the 
	more conventional $x$ and $y$, we obtain the following as an equation of the curve $X_0(T(T+1))$ in the affine plane $\Spec(\Ci[x, y])$:
	$$
		\frac{x^{q+1}}{T+x} = 	\frac{y^{q+1}}{(T+1)+y}.  
	$$
\end{example}


\section{Modular forms over $\F_q(T)$}\label{sDMF}

In the context of Drinfeld modular varieties there are two different concepts that generalize 
classical modular forms. The first are Drinfeld modular forms, which are $\Ci$-valued holomorphic 
functions on $\Omega^r$. The second are Drinfeld automorphic forms, which are $\C$-valued functions 
on certain adele groups; these automorphic forms have a combinatorial interpretation as harmonic cochains on the Bruhat-Tits building $\cB^r$ of $\PGL_r(\Fi)$. Both of these concepts are important in the study of Drinfeld modular varieties. 


\subsection{Drinfeld modular forms}\label{subsection-DMF}

Our main references are \cite{Goss-ModForms}, \cite{Gekeler-LNM}, \cite{Gekeler-Coeffs}, \cite{GekelerJNTB}, \cite{BBP}. 
Let $r\geq 2$ and let $\G$ be a congruence subgroup of $\GL_r(A)$, i.e., $\G(\fn)\subseteq \G\subseteq \GL_r(A)$ for some $\fn\neq 0$. If we normalize the projective coordinates of $\boz=(z_1, \dots, z_r)\in \Omega^r$ by $z_r=1$, then 
$\gamma=(\gamma_{m,n})\in \G$ acts on $\Omega^r$ as 
$$
\gamma \boz = j(\gamma, \boz)^{-1} (\ldots, \sum_n \gamma_{m, n} z_n, \dots ),
$$
where $j(\gamma, \boz) = \sum_{n=1}^r \gamma_{r, n}z_n$. For example, if $r=2$ and we identify  $\Omega^2$ with $\Ci - \Fi$, then for $\gamma = \begin{pmatrix}a&b\\c&d\end{pmatrix}\in\G$ we have $j(\gamma, \boz)=j(\gamma,z) = cz+d$ and $\gamma \boz = \gamma z = \dfrac{az+b}{cz+d}$.

\begin{defn}
A \textit{Drinfeld modular form} for $\G$ of weight $k\in \Z_{\geq 0}$ and type $m\in \Z/(q-1)\Z$ is a holomorphic function 
$$
f\colon \Omega^r\To \Ci
$$
such that 
\begin{enumerate}
	\item[(i)] $f(\gamma \boz) =(\det \gamma)^{-m} j(\gamma, \boz)^k f(\boz)$ for all $\gamma\in \G$, and 
	\item[(ii)] $f(\boz)$ is holomorphic at the cusps of $\G$. 
\end{enumerate}
Denote the space of such functions by $M^r_{k, m}(\G)$. (It is shown in \cite{BBP} that $M^r_{k, m}(\G)$ is finite dimensional over $\Ci$.)
\end{defn}
Condition (ii) is technically complicated to explain when $r\geq 3$, so we will 
explain it only for $r=2$, and refer to \cite{BBP} for $r\geq 2$. Because $\G$ is a congruence group, it contains the subgroup $U_b=\left(\begin{smallmatrix} 1 & bA \\ 0 & 1\end{smallmatrix}\right)$ for 
some nonzero $b\in A$. Condition (i) implies that $f(z+b)=f(z)$, which itself implies that $f(z)$ can be expanded as 
$$
f(z)=\sum_{n\in \Z} a_n (1/e_{bA}(z))^n, \qquad a_n\in \Ci, 
$$
assuming $\Im(z)\colonequals \inf_{\alpha\in \Fi}\abs{z-\alpha}\gg 0$ (for simplicity, we will omit this condition in what follows). We say that $f(z)$ is \textit{holomorphic at the cusp $[\infty]$} if in the above expansion $a_n=0$ for all $n<0$ (this vanishing of coefficients with negative indices does not depend on the choice of~$b$). 
Next, for $g\in \GL_2(A)$, put $f|_g (z) = (\det g)^m j(g,z)^{-k}f(g z)$. This 
$f|_g$ satisfies (i) for any $\gamma \in g^{-1}\G g$, which is again a congruence group. Condition (ii) means that $f|_g$ is holomorphic at $[\infty]$ for all $g\in \GL_2(A)$. 
(Note that $f|_g=f$ for $g\in \G$, so for this last condition to hold it suffices that $f|_g$ is holomorphic at $[\infty]$ for left coset representatives of $\G$ in $\GL_2(A)$.) 

\begin{example}
Important examples of modular forms for $\GL_r(A)$ arise as ``coefficient forms". For 
$\boz=(z_1, \dots, z_r)\in \Omega^r$, we define the rank-$r$ lattice 
$
\La_{\boz}=Az_1+Az_2+\cdots+Az_r\subset \Ci$. 
Denote the Drinfeld module of rank $r$ associated to $\La_{\boz}$ by $\phi^{\boz}$. It is determined by 
\begin{equation}\label{eqCoeffForms}
	\phi^{\boz}_T=Tx+g_1(\boz)x^q+\cdots+g_r(\boz)x^{q^r}. 
\end{equation}
The functions $g_i(\boz)$, $1\leq i\leq r$, are Drinfeld modular forms for $\GL_r(A)$ of type $0$ and weights $q^i-1$. The function $\Delta_r(\boz)\colonequals g_r(\boz)$, called the \textit{Drinfeld discriminant function}, plays an especially important role in the study of the cuspidal divisor group in this context; it is the analogue of Ramanujan's $\Delta$ function. Note that $\Delta_r(\boz)$ 
is non-vanishing on $\Omega^r$ since $g_r(\La)\neq 0$ for any $A$-lattice $\La$ of rank $r$. 

The Eisenstein series $E_{q^n-1}(\boz)\colonequals E_{q^n-1}(\La_{\boz})$ defined in \eqref{eqEisSer} and the coefficients $e_{n}(\boz)\colonequals e_{n}(\La_{\boz})$ defined in \eqref{eqCoefExp} 
are Drinfeld modular forms of weight $q^n-1$.  
It can be shown that the $\Ci$-algebra of all Drinfeld modular forms of type $0$ is a polynomial ring:
$$
\bigoplus_{k\geq 1} M^r_{k, 0}(\GL_r(A)) = \Ci[g_1, \dots, g_r] = \Ci[e_1, \dots, e_r]= \Ci[E_{q-1}, E_{q^2-1}, \dots, E_{q^r-1}] .
$$
\end{example}

\begin{example}
    The \textit{$j$-function} on $\Omega^2$ is $j(z)\colonequals g_1(z)^{q+1}/g_2(z)$. This function is holomorphic on $\Omega$ but has a pole at the cusp $[\infty]$. 
Since $j(\gamma z)=j(z)$ for all $\gamma\in \GL_2(A)$, it defines a rational function on $X_0(1)$. 
In fact, $j(z)$ generates the field of rational functions on $X_0(1)\cong \p^1_{\Ci}$. 
\end{example}

Now we specialize to $r=2$ and $\G=\G_0(\fn)$. To simplify the notation, we will omit the superscript $r$, so for example $\Omega\colonequals\Omega^2$ and $M_{k,m}(\fn) \colonequals M^2_{k,m}(\G_0(\fn))$. Since $\G_0(\fn)$ contains the group of scalar matrices $\left(\begin{smallmatrix} \alpha & 0 \\ 0 & \alpha\end{smallmatrix}\right)$, $\alpha\in \F_q^\times$, applying 
condition (i) to such matrices one concludes that if $M_{k, m}(\fn)\neq 0$, then $k\equiv 2m\Mod{q-1}$. Hence, if $q$ is odd and $M_{k, m}(\fn)\neq 0$, then $k$ is necessarily even and $m=k/2$ or $m=k/2+(q-1)/2$ modulo $q-1$. Next, a simple calculation shows that the differential $\dd z$ on $\Omega$ satisfies 
$$
\dd(\gamma z) = \frac{\det(\gamma)}{(cz+d)^2}\dd z \quad \text{for all }\gamma\in \GL_2(\Fi). 
$$
Hence, if $f(z)\in M_{2k, k}(\fn)$, then $f(z)(\dd z)^k$ can be identified with a $k$-fold differential form on the Drinfeld modular curve $X_0(\fn)$. 

Since $\Gamma_0(\fn)$ contains the subgroup $U_1$, the expansion of a Drinfeld modular form $f$ at $[\infty]$ is $f(z)=\sum_{n\geq 0} a_n (1/e_A(z))^n$. Instead we will now use the parameter
$$t(z) = \frac{1}{\pi_C e_A(z)} = \frac{1}{e_C(\pi_C z)} = \frac{1}{\pi_C} \sum_{a\in A} \frac{1}{z+a},$$
which plays the role of the classical parameter $q = e^{2i\pi z}$ (this renormalization is chosen so that the expansions of the modular forms $\pi_C^{1-q} g_1$ and $\pi_{C}^{1-q^2} g_2$ at~$[\infty]$ have coefficients in~$A$). Then we have $f = \sum_{n\geq 0} a'_n t^n$ with $a'_n \in \Ci$. Since $t(\alpha z)=\alpha^{-1}t(z)$  for any $\alpha \in \F_q^\times$, the coefficients $a'_n$ are zero unless $n\equiv m \Mod{q-1}$ so the expansion of~$f$ is $$f = \sum_{i\geq 0} b_i t^{m+i(q-1)}.$$
A Drinfeld modular form is said to be \emph{cuspidal} (resp. \emph{doubly cuspidal}) \emph{at the cusp $[\infty]$} if $a'_0=0$ (resp. $a'_0 = 0$ and $a'_1=0$) (again these vanishing conditions do not depend on the choice of $b$).  If for all $g\in \GL_2(A)$, $f|_g$ is cuspidal (resp. doubly cuspidal) at $[\infty]$, we say that the Drinfeld modular form $f$ is \emph{cuspidal} (resp. \emph{doubly cuspidal}). Let $M_{k,m}^0(\fn)$ (resp. $M_{k,m}^{0,0}(\fn)$) denote the subspaces of such functions. Since $\dd e_A(z) / \dd z=1$, we have $\dd t = -\pi_C t^2 \dd z$ so doubly cuspidal Drinfeld modular forms play a role similar to classical cusp forms. Namely the map $f(z) \mapsto f(z) \dd z$ is an isomorphism between $M_{2,1}^{0,0}(\fn)$ and the space of holomorphic differential forms on the curve $X_0(\fn)$. In particular the dimension of $M_{2,1}^{0,0}(\fn)$ is equal to the genus of $X_0(\fn)$.


Hecke operators on Drinfeld modular forms can be defined as double coset operators for $\Gamma_0(\fn)$. Let $\fm \lhd A$ be a nonzero ideal of $A$. For the double coset $\Gamma_0(\fn) \bs (\Gamma_0(\fn) \left(\begin{smallmatrix}\fm&0\\0&1\end{smallmatrix}\right) \Gamma_0(\fn))$, a set of representatives is  \begin{equation}\label{eq-Sm}
S_{\fm}=\{ \left(\begin{smallmatrix} a & b \\ 0 & d \end{smallmatrix}\right) \in M_2(A) \, :\, a,d \text{ monic}, (ad) = \fm, (a)+\fn = A, \deg b < \deg d\}.
\end{equation} For $f \in M_{k,m}(\fn)$, we define $f|T_\fm = \sum_{g\in S_{\fm}} f|_g$. In more concrete terms
$$f|T_{\fm}\, (z) = \frac{1}{\fm^{k-m}} \sum_{\left(\begin{smallmatrix}a&b\\0&d\end{smallmatrix}\right) \in S_\fm} f\left( \frac{az+b}{d} \right).$$ 
The \emph{$\fm$-th Hecke operator $T_\fm$} is a $\Ci$-linear transformation of the space $M_{k,m}(\fn)$. It stabilizes the cuspidal and doubly-cuspidal subspaces. The Hecke operators $(T_\fm)_{\fm\lhd A}$ generate a commutative $\Ci$-subalgebra 
of $M_{k,m}(\fn)$, called the \emph{Hecke algebra} for $M_{k,m}(\fn)$. On Drinfeld modular forms, Hecke operators are completely multiplicative, i.e. for any $\fm,\fm' \lhd A$, we have $T_\fm T_{\fm'} = T_{\fm \fm'}=T_{\fm'} T_\fm$. This property distinguishes them from Hecke operators on classical modular forms, where they are only multiplicative.

Another important difference is that, the characteristic being positive, the space $M_{k,m}^{0,0}(\fn)$ has no inner product. Hence there is no guarantee that $T_\fm$ is diagonalizable. Goss was the first to point out this problem. Since then, examples of non-diagonalizable Hecke operators have been obtained in special cases (cf. \cite{Li-Meemark} for the subgroups $\Gamma_1(T)$ and $\Gamma(T)$) but in general, the question is still wide open. In particular we know no natural bases of Drinfeld modular forms which are simultaneous eigenforms for Hecke operators.


\subsection{Harmonic cochains}\label{subsection-HC}
Let $r\geq 2$ and let $V$ be an $r$-dimensional vector space over $\Fi$. 
Denote by $\cO_\infty$ the ring of integers of $\Fi$ and let $\pi_\infty$ be a uniformizer of $\cO_\infty$. 
An \textit{$\cO_\infty$-lattice in $V$} is a free $\cO_\infty$-module of rank $r$ which contains a basis of $V$. Two lattices $L_1$ and $L_2$ are \textit{homothetic} if there exists $\alpha\in \Fi^\times$ with $\alpha\cdot L_1=L_2$; this defines an equivalence relation on the set of lattices in $V$. We denote the equivalence class of $L$ by $[L]$. The \textit{Bruhat-Tits building} of $\PGL_r(\Fi)$ is the $(r-1)$-dimensional simplicial complex $\sB^r$ with set of vertices  
$$
\Ver(\sB^r)=\{[L]\mid L \text{ is a lattice in }V\},  
$$
in which the vertices $[L_0], \cdots, [L_n]$ form an $n$-simplex if and only if there is $L_i'\in [L_i]$, $1\leq i\leq n$, such that 
$$
L_0'\supsetneq L_1'\supsetneq \cdots \supsetneq L_n'\supsetneq \pi_\infty L_0'. 
$$
Since $\GL_r(\Fi)$ acts transitively on $\Ver(\sB^r)$ and the stabilizer of a vertex is conjugate to 
$\Fi^\times \GL_r(\cO_\infty)$, we have a bijection 
$$\Ver(\sB^r) \cong \GL_r(\Fi)/\Fi^\times \GL_r(\cO_\infty).$$ Similar bijections exist between the sets of higher dimensional 
simplices of $\sB^r$ of various types and left cosets of parahoric subgroups of $\GL_r(\cO_\infty)$. 
When $r=2$, $\sB^2$ is an infinite tree in which every vertex is adjacent to exactly $q+1$ other vertices. 

Harmonic cochains on simplicial complexes naturally arise in the study of combinatorial laplacians; cf. \cite{Garland}. A variant of harmonic $i$-cochains on $\sB^r$ was defined by de Shalit \cite{deShalit}: these are functions on pointed $i$-simplices of $\sB^r$, $1\leq i\leq r-1$, satisfying certain conditions (in \cite{deShalit} the building $\sB^r$ is defined over an arbitrary local field). The significance of the group $\Har^i(\sB^r, \Q_\ell)$ 
of $\Q_\ell$-valued harmonic $i$-cochains is that it is isomorphic to the $\ell$-adic cohomology group $H^i_{\mathrm{et}}(\Omega^r, \Q_\ell)$ of $\Omega^r$; cf.\cite{SS}, \cite{deShalit}. Also, $\Har^i(\sB^r, \Q_\ell)$ can be interpreted as a space of automorphic forms on the adele group $\GL_r(\A_F)$ with special restriction at $\infty$; see \cite{AitAmrane} (one transforms functions on $\GL_r(\A_F)$ to functions on $\GL_r(\Fi)$ using the strong approximation theorem and uses the bijections between the sets of simplices of $\sB^r$ and cosets in $\PGL_r(\Fi)$ mentioned earlier). 

The most relevant for our purposes are the harmonic $1$-cochains, which for $r=2$ were already defined by van der Put in \cite{vdPut}.  
\begin{defn} 
Let $\Ed(\sB^r)$ be the set of oriented $1$-simplices of $\sB^r$. Let $R$ be a commutative ring with unity. 
For $r=2$, an \emph{$R$-valued harmonic $1$-cochain} on $\sB^2$ is a 
function $f\colon \Ed(\sB^2)\rightarrow R$ that satisfies
\begin{enumerate}
    \item $$f(e)+f(\overline{e})=0\text{ for all }e\in \Ed(\sB^2).$$ 
    \item $$\sum_{\substack{e\in \Ed(\sB^2)\\o(e)=v}}f(e)=0\text{ for all }v\in \Ver(\sB^2).$$ 
\end{enumerate}
Here, $o(e)$ is the origin of $e$ and $\bar{e}$ is the edge $e$ with opposite orientation. 

When $r\geq 3$, two of the conditions defining harmonic $1$-cochains are similar to (1) and (2), with (2) refined by ``types" of edges, where $\mathrm{type}([L_0], [L_1])=\dim_{\F_q}L_0/L_1$.  
But there are two extra conditions. One of these conditions says that the sum of values of $f$ over the edges of a closed path in $\sB^r$ is $0$, and the other that the values of $f$ are uniquely determined by its values on edges of type $1$. We denote the space of $R$-valued harmonic $1$-cochains by $\Har^1(\sB^r, R)$. 
\end{defn}

From another perspective, $\sB^r$ is a combinatorial ``skeleton'' of $\Omega^r$. In fact, there is a $\GL_r(\Fi)$-equivariant map $\Omega^r\to \sB^r$ which sends affinoids to simplices; see \cite{Drinfeld}. The following important result relates the group of holomorphic invertible functions $\cO(\Omega^r)^\times$ on $\Omega^r$ to $\Z$-valued harmonic $1$-cochains on $\sB^r$: 
\begin{equation}\label{eqGvdP}
    0 \To \Ci^\times \To \cO(\Omega^r)^\times \xrightarrow{\mathrm{dlog}} \mathrm{Har}^1(\sB^r, \Z)\To 0, 
\end{equation}
where $\mathrm{dlog}$ is some sort of a logarithmic derivative.
The existence of \eqref{eqGvdP} was proved for $r=2$ by van der Put and extended to arbitrary $r\geq 2$ by Gekeler \cite{GekelerANT}.

From now on, the discussion will mostly concentrate on harmonic $1$-cochains on the Bruhat-Tits tree $\sB^2$. To simplify the notation, we denote $\sT=\sB^2$. The group $\GL_2(\Fi)$ acts on $\sT$ from the left.  Let 
$$
\cH(\fn, R)\colonequals \Har^1(\sT, R)^{\G_0(\fn)}
$$
be the submodule of $R$-valued harmonic $1$-cochains such that $f(\gamma e)=f(e)$ for all $\gamma\in \G_0(\fn)$ and $e\in \Ed(\sT)$. 
   The module of $R$-valued $\G_0(\fn)$-invariant \textit{cuspidal harmonic cochains}, denoted  $\cH_0(\fn, R)$, is the submodule of $\cH(\fn, R)$ consisting of functions which have compact 
   support as functions on $\G_0(\fn)\bs \sT$, i.e., functions which have value $0$ on all but finitely many edges of $\G_0(\fn)\bs \sT$. We denote by $\cH_{00}(\fn, R)$ the image 
   of $\cH_0(\fn)\otimes R$ in $\cH_0(\fn, R)$. (It is easy to construct examples where the inclusion $\cH_{00}(\fn, R)\subseteq \cH_0(\fn, R)$ is strict; cf. \cite[$\S$1.1]{papikian_eisenstein_2015}.)
It is known that the quotient 
graph $\G_0(\fn)\bs \sT$ is the edge disjoint union 
$$
\G_0(\fn)\bs \sT = (\G_0(\fn)\bs \sT)^0\cup \bigcup_{s\in \G_0(\fn)\bs \p^1(F)} h_s
$$
of a finite graph $(\G_0(\fn)\bs \sT)^0$ with a finite number of half-lines $h_s$, called \textit{cusps}. The cusps are in bijection with the orbits of the natural (left) action of $\G_0(\fn)$ on $\p^1(F)$.  
It is clear that $f\in \cH(\fn, R)$ is cuspidal if and only if it eventually vanishes on each $h_s$. 
One can show that $\cH_0(\fn, \Z)$ and $\cH(\fn, \Z)$ are finitely generated free 
$\Z$-modules of rank $g(\fn)$ and  $g(\fn)+c(\fn)-1$, respectively, where $g(\fn)$ 
is the genus of the curve $X_0(\fn)$ and $c(\fn)$ is the number of cusps. 

\begin{rem}
Drinfeld modular forms are related to $\Ci$-valued harmonic cochains. More precisely, 
using residues of differential forms on $\Omega$, Teitelbaum \cite{Teitelbaum} established an isomorphism between $M_{2,1}^0(\fn)$ and $\cH_0(\fn, \Ci)$. This isomorphism restricts to an 
isomorphism $M_{2,1}^{00}(\fn) \overset{\sim}{\To} \cH_{00}(\fn, \Ci)$; see \cite[(6.5)]{GekelerREVERSAT}. 
\end{rem}

Harmonic $1$-cochains invariant under congruence groups have Fourier expansions. This theory for $r=2$ was first developed by Weil in adelic language, and recast into more explicit formulas by Gekeler \cite{GekelerJNTB2} and P\'al \cite{PalDoc}. The theory was extended to arbitrary $r\geq 2$ in \cite{PW-MA}. We briefly discuss the $r=2$ case. 

The edges of $\sT$ are in bijection with $\GL_2(\Fi)/\Fi^\times \cI_\infty$, where $\cI_\infty$ is the Iwahori group:
$$
\cI_\infty:=\left\{\begin{pmatrix} a & b\\ c & d\end{pmatrix}\in \GL_2(\cO_\infty)\ \bigg|\ c\in \pi_\infty\cO_\infty\right\}. 
$$
Let 
$$\Ed(\sT)^+:=\left\{\begin{pmatrix} \pi_\infty^k & u \\ 0 & 1\end{pmatrix}\ \bigg|\
\begin{matrix} k\in \Z\\ u\in \Fi,\ u\mod{\pi_\infty^k\cO_\infty}\end{matrix}\right\}. $$
It is not hard to show that $\Ed(\sT)=\Ed(\sT)^+\bigsqcup \Ed(\sT)^+ \begin{pmatrix} 0 & 1\\ \pi_\infty & 0\end{pmatrix}$.

Assume that $R$ is a ring such that $p\in R^\times$ and $R$ contains the $p$-th roots of unity. A function $f\in \cH(\fn, R)$ has a \textit{Fourier expansion} 
$$
f \left(\begin{pmatrix} \pi_\infty^k & u \\ 0 &1\end{pmatrix}\right) = f^0(\pi_\infty^k)+
\sum_{0\leq j\leq k-2} q^{-k+2+j}\sum_{\substack{\fm\in A\text{, monic}\\ \deg(\fm)=j}}f^\ast(\fm)\, \nu(\fm u), 
$$
where $\nu(x):=-1$ if $x$ has a term of order $\pi_\infty$ in its $\pi_\infty$-expansion; $\nu(x):=q-1$ otherwise. Here the \textit{Fourier coefficients}  $f^0(\pi_\infty^k)$ and $f^\ast(\fm)$ are certain finite sums of values of $f$ twisted by a character $\chi\colon \Fi\to \C^\times$ taking values in the $p$-th roots of unity. Moreover, the \textit{constant Fourier coefficient} $f^0(\pi_\infty^k)$ 
is equal to $0$ for cuspidal harmonic cochains.

Following Weil, one can attach a Dirichlet $L$-series to a cuspidal harmonic cochain $f \in \cH_0(\fn,\C)$. Put
$$L (f,s) = \sum_{\fm} f^*(\fm) |\fm |^{-s-1}$$
where the sum is over all non-negative  divisors on $F$, including those with an $\infty$-component. This $L$-series is defined for $s\in\C$, is a polynomial in $q^{-s}$, and satisfies a functional equation when substituting $s \mapsto 2-s$.

Fix a non-zero ideal $\fn\lhd A$. Given a non-zero ideal $\fm\lhd A$, define 
an $R$-linear transformation of the space of $R$-valued functions on $\Ed(\sT)$ by 
\begin{equation}\label{eqDefTm}
f|T_\fm:=\sum_{g\in S_{\fm}} f|_g
\end{equation}
where the set $S_{\fm}$ has been defined in \eqref{eq-Sm}. 
This transformation is the \textit{Hecke operator} $T_\fm$. Following a common convention, 
for a prime divisor $\fp$ of $\fn$ one sometimes writes $U_\fp$ instead of $T_\fp$.

The group-theoretic proofs of properties of the Hecke operators 
acting on classical modular forms work also in this setting. In particular, 
the Hecke operators preserve $\cH_0(\fn, R)$, and satisfy the recursive formulas: 
\begin{align*}
T_{\fm\fm'}&= T_\fm T_{\fm'}\quad \text{if}\quad  \fm+\fm'=A,\\
T_{\fp^i} &= T_{\fp^{i-1}}T_\fp-|\fp|T_{\fp^{i-2}}\quad \text{if}\quad  \fp\nmid \fn, \\
T_{\fp^i} &= T_\fp^i\quad \text{if}\quad  \fp| \fn. 
\end{align*}
Let $\T(\fn)$ be the commutative $\Z$-subalgebra of $\End_\Z(\cH_0(\fn, \Z))$ 
generated by all Hecke operators. 

One can give an explicit formula the action of $T_\fm$ on the Fourier expansion of $f\in \cH_0(\fn, R)$. This formula implies that 
\begin{equation}\label{eqfTm}
    (f|T_\fm)^\ast(1)=|\fm|f^\ast(\fm). 
\end{equation}
Thus, the pairing 
\begin{equation}\label{eqPairing}
\begin{array}{ccc}
 \T(\fn)\times \cH_0(\fn, \Z) & \To & \Z \\
 ( T, f) & \longmapsto & (f|T_\fm)^\ast(1) 
 \end{array}
\end{equation}
is non-degenerate and becomes perfect after tensoring with $\Z[p^{-1}]$.

\begin{example}\label{exampleEis} The Gekeler--van der Put map \eqref{eqGvdP}, when applied to modular units arising from the Drinfeld discriminant function, produces Eisenstein harmonic $1$-cochains. This is a reflection of the Kronecker limit formula in this context; cf. \cite{PalDoc}, \cite{WeiInventiones}. We give an example of a special case, where the Eisenstein harmonic cochain can be described quite explicitly.  

Let $\fp\lhd A$ be a prime of degree $3$. In this case, the quotient graph $\G_0(\fp)\bs \sT$ looks like the graph in Figure~\ref{Fig1}; cf. \cite[Section 4.1]{papikian_eisenstein_2016}. 

\begin{figure}[ht]
\begin{tikzpicture}[scale=2, ->, >=stealth, semithick, inner sep=.5mm, vertex/.style={circle, fill=black}]

\node[vertex] (00) at (0, 0) {};
\node[vertex] (01) at (0, 1) {};
\node[vertex] (11) at (1, 1) {};
\node[vertex] (10) at (1, 0) {};
\node[vertex] (c1) at (-1, 0) {};
\node[vertex] (c2) at (2, 0) {};
\node at (.5, 1.05) {$\vdots\ b_u$};

\path[]
(11) edge[bend right]  (01)
(11) edge[bend left]  (01)
(01) edge node[auto,swap] {$a_\infty$} (00)
(10) edge node[auto] {$d_\infty$} (00)
(11) edge node[auto] {$a_1$} (10)
(00) edge[dashed] node[auto] {$s_\infty$} (c1)
(c2) edge[dashed] node[auto] {$s_1$} (10);
\end{tikzpicture}
\caption{$\G_0(\fp)\bs \sT$, $\fp$ is a prime of degree $3$.}\label{Fig1}
\end{figure}
The matrix representatives of the edges are the following. The dashed edges
$$
s_\infty=\begin{pmatrix} \pi_\infty & 0 \\0&1\end{pmatrix}, \quad s_1=\begin{pmatrix} \pi_\infty^3 & 0 \\0&1\end{pmatrix}
$$
indicate that they are the first edges on a half-line corresponding to the cusps $[\infty]$ and $[0]$, respectively; 
$$
a_\infty=\begin{pmatrix} \pi_\infty^2 & \pi_\infty \\0&1\end{pmatrix}, \quad a_1=\begin{pmatrix} \pi_\infty^3 & \pi_\infty^2 \\0&1\end{pmatrix},
\quad
d_\infty=\begin{pmatrix} \pi_\infty^2 & 0 \\0&1\end{pmatrix}; 
$$
there are $q$ edges 
$$
b_u = \begin{pmatrix} \pi_\infty^3 & \pi_\infty + u \pi_\infty^2 \\0&1\end{pmatrix}, \quad u \in \F_q.
$$ 

Let $E\colonequals \mathrm{dlog}(\Delta/\Delta_{\fp})\in \cH(\fp, \Z)$, where $\Delta = \Delta_2$ and $\Delta_{\fp}(z) :=\Delta(\fp z)$. One can compute the values of $E$ on the edges of $\G_0(\fp)\bs \sT$ using \cite[Corollary 2.9]{gekeler_1997} and \cite[Lemma 2.4]{ho_rational_2024}:
\begin{enumerate}
    \item $E(s_\infty) = E(s_1) = (q^2+q+1)(q-1)^2$.
    \item $E(a_\infty) = E(\overline{a_1}) = q(q-1)^2$ and $E(d_\infty) = (2q+1)(q-1)^2$.
     \item $E(b_u) = (q-1)^2$ for $u\in \F_q$.
\end{enumerate}
By \cite[Corollary 2.11]{gekeler_1997} and \cite[Lemma 4.4]{ho_torsion_2024}, we have $$E|U_{\fp}=E, \quad \text{and}\quad E|T_\fq = (|\fq|+1)E\quad \text{for all prime $\fq\neq \fp$}.$$ By using this and applying \cite[(2.5) and Corollary 2.8]{gekeler_1997} and \cite[Lemma 3.6]{ho_torsion_2024}, one can compute the Fourier coefficients of $E$:
\begin{enumerate}
    \item $E^0(\pi_\infty^k) = q^{1-k} (q^2+q+1)(q-1)^2$.
    \item $E^\ast(1) = q^{-1}(q+1)(q-1)^2$.
    \item $E^\ast(\fm) = \frac{(q+1)(q-1)^2}{q\, |\fm|} \, \prod_{i=1}^s \frac{|\fq_i|^{k_i+1}-1}{|\fq_i|-1}$ for $\fm = \fp^k \, \prod_{i=1}^s {\fq_i}^{k_i}\lhd A$. 
\end{enumerate}
Note that, similar to the appearance of divisor function in the Fourier expansions of classical Eisenstein series \cite[p. 100]{Miyake}, the factor  $(\abs{\fq}^{k+1}-1)/(\abs{\fq}-1)=\abs{\fq}^k+\cdots+\abs{\fq}+1$ in the expression for $E^\ast(\fm)$ is a version of the divisor function for $\fq^{k+1}$. We also point out that the term $(q^2+q+1)$ in 
$E^0(\pi_\infty^k)$ is the order of the cuspidal divisor group of $X_0(\fp)$, and this type of relations are  important in the proofs in Section \ref{sDJT=C}.
\end{example}

The analogue of the Petersson inner product in this setting is the pairing on $\cH_0(\fn, \C)$ 
defined by 
$$
(f, g) =\sum_{e\in \Ed(\G_0(\fn)\bs \sT)} f(e)\overline{g(e)} \mu(e)^{-1},
$$
where $\mu(e)=\frac{q-1}{2}\# \Stab_{\G_0(\fn)}(\tilde{e})$ and $\tilde{e}$ 
is a preimage of $e$ in $\sT$. (A Haar measure on $\GL_2(\Fi)$ induces a push-forward measure 
on $\Ed(\G_0(\fn)\bs \sT)$, which, up to scalar multiple, is equal to $\mu(e)$.) The Hecke operators $T_\fm$, $(\fn, \fm)=1$, are self-adjoint with respect to this pairing. Hence, the usual conclusions about the Hecke operators being diagonalizable and their eigenvalues being real are also valid in this setting. 

The Hecke operators may also be defined using correspondences on $X_0(\fn)$. Recall the moduli interpretation of $Y_0(\fn)$: it is the generic fibre of the coarse moduli scheme for pairs $(\phi,G_\fn)$ consisting of a Drinfeld $A$-module $\phi$ of rank~$2$ over $F$ with an $F$-rational $A$-submodule $G_\fn$ of $\phi[\fn]$ isomorphic to $A/\fn$ ($F$-rational means that $\sigma(G_\fn)=G_\fn$ for all $\sigma\in \Gal(F^\alg/F)$). For $\fm \lhd A$, the Hecke operator $T_\fm$ is defined as the correspondence on $Y_0(\fn)$ given by
$$T_\fm : (\phi,G_\fn) \mapsto \sum_{G_\fm \cap G_\fn = \{ 0 \} } (\phi/G_\fm , (G_\fn + G_\fm)/G_\fm).$$
It uniquely extends to $X_0(\fn)$. The resulting correspondence induces an endomorphism
of the Jacobian variety $J_0(\fn)$ of $X_0(\fn)$, also denoted $T_\fm$. The $\Z$-subalgebra of $\End(J_0(\fn))$ generated by $(T_{\fm})_{\fm \lhd A}$ 
is canonically isomorphic to $\T(\fn)$; this is a 
consequence of Drinfeld's Reciprocity Law \cite[Theorem 2]{Drinfeld}. Having this fact, 
one can use the classical Shimura construction to 
associate to a $\T(\fn)$-eigenform $f\in \cH_0(\fn, \C)$ an abelian variety $B_f$ 
whose number of points over $\F_\fp$, $\fp\nmid \fn$, is computed using the eigenvalue of $T_\fp$ acting on $f$. In particular, if $f$ has rational eigenvalues, then $B_f$ 
is an elliptic curve over $F$. Combining this with some deep results of Grothendieck and Deligne, 
one can deduce the following analogue of the Modularity Theorem (cf. \cite[(8.3)]{GekelerREVERSAT}):

\begin{thm}
    Let $E$ be an elliptic curve over $F$ with split multiplicative reduction at $\infty$. Then there is a non-constant morphism $X_0(\fn)\to E$ defined over $F$, where $\fn$ is the finite part of the conductor of $E$. 
\end{thm}

\begin{defn}\label{defEI} The \textit{Eisenstein ideal} $\fE(\fn)$ of $\T(\fn)$ is the 
ideal generated by the elements $T_\fp-(|\fp|+1)$ for all prime $\fp\nmid \fn$. 
\end{defn}

The quotient $\T(\fn)/\fE(\fn)$ is a finite ring. Indeed, otherwise the perfectness of the pairing \eqref{eqPairing} implies that there is nonzero $f\in \cH_{00}(\fn, \Q)$ annihilated by $\fE(\fn)$, 
which contradicts Weil's bounds using the Shimura construction mentioned earlier. 

\begin{rem}
\begin{enumerate}
    \item The Eisenstein ideal in the context of Drinfeld modular curves was first defined by Tamagawa \cite{Tamagawa95} as the kernel of $\T(\fp)\to \End_\Z(\cC(\fp))$, where $\fp$ is prime and $\cC(\fp)$ is the cuspidal divisor group of $J_0(\fp)$. It can be shown that this is equivalent to  Definition \ref{defEI} when $\fn=\fp$ is prime; cf. \cite{papikian_eisenstein_2016}. Moreover, $U_\fp-1\in \fE(\fp)$, so $\fE(\fp)$ 
    is the analogue of Mazur's definition of the Eisenstein ideal in \cite{Mazur1977}. 
    \item Depending on the problem where it is used, the definition of Eisenstein ideal 
    might be modified to include elements of the form $U_\fp+a$ or $W_\fp+a$, where $\fp\mid \fn$, $a\in \Z$,  
    and $W_\fp$ is an Atkin-Lehner involution; cf. \cite{papikian_eisenstein_2016}, \cite{papikian_rational_2017}, \cite{ho_rational_2024}. 
\end{enumerate}
\end{rem}

\section{Torsion of Drinfeld modules}\label{sDMT}

\subsection{Analogue of Ogg's conjecture}
Let $\phi$ be a Drinfeld $A$-module of rank $2$ over $F$. Its torsion $A$-module $(^\phi F)_\tor\colonequals \bigcup_{a\in A} \phi[a](F)$ is finite as can be proved, for instance, by an argument involving the reductions of $\phi$; cf. Remark \ref{rem5.1} (2). 
Moreover, $(^\phi F)_\tor$ can be generated by two elements $$(^\phi F)_\tor \cong A/\fm \oplus A/\fn,$$ where $\fm$ and $\fn$ are non-zero ideals of $A$ and $\fm$ divides $\fn$. For rank-$2$ Drinfeld modules we have a counterpart of Conjecture~\ref{conjOgg1} on the possible torsion subgroups for elliptic curves.

\begin{conj*}{TDM}[{\cite[Conjecture 1]{SchweizerMZ}}] \label{conjSchweizer1}
Let $\phi$ be a rank-$2$ Drinfeld $A$-module over $F$. If we write $$(^\phi F)_\tor\cong A/\fm\oplus A/\fn\text{~with~}\fm|\fn,$$ then $\deg \fm+\deg \fn\leq 2$. This is essentially equivalent to saying that if $\fp$ is a prime of $A$ of degree $\geq 3$, then the Drinfeld modular curve $Y_1(\fp)$ has no $F$-rational points.
\end{conj*}

\begin{rem}\label{rem5.1}
    \begin{enumerate}
        \item The finite $A$-modules listed in Conjecture \ref{conjSchweizer1} as possible $F$-rational torsion submodules of rank-$2$ Drinfeld modules correspond exactly to the levels of those Drinfeld modular curves $X_1(\fn)$ and $X(\fn)$ which are isomorphic to $\p^1_F$. Thus, these finite modules occur as the torsion submodules of infinitely many non-isomorphic Drinfeld modules.   
        \item As of 2024, Conjecture \ref{conjSchweizer1} remains largely open. On the other hand,  under an extra assumption, it is easy to prove. We say that the Drinfeld module $\phi$ defined over $F$ by $\phi_T(x)=Tx+g_1x^q+\cdots+g_r x^{q^r}$ has \textit{good reduction} at the prime $\fl\lhd A$  if 
        $\ord_{\fl}(g_i)\geq 0$ for $1\leq i\leq r$ and $\ord_{\fl}(g_r)=0$. If $\phi$ has good reduction at $\fl$, then reducing the coefficients of $\phi_T$ modulo $\fl$, one obtains a Drinfeld module $\bar{\phi}$ over $\F_\fl$ of the same rank. It is not hard to show that the reduction modulo $\fl$ induces an injection $({^\phi}F)_\tor\hookrightarrow ({^{\bar{\phi}}}\F_\fl)_\tor$; cf. \cite[Theorem 6.5.10]{PapikianGTM}. Thus, $\# ({^\phi}F)_\tor\leq \abs{\fl}$. In particular, if after replacing $\phi$ by an isomorphic Drinfeld module it acquires good reduction at a prime of degree $\leq 2$, then Conjecture \ref{conjSchweizer1} holds for $\phi$.  
        \item It seems like an interesting and important problem to give an explicit conjectural classification of possible $F$-rational torsion submodules of Drinfeld modules of rank $3$,  possibly using the geometry of Drinfeld modular surfaces as a guide.  
    \end{enumerate}
\end{rem}

Following Ogg's philosophy on the existence of rational points on modular curves, we may also formulate the following conjecture which is an analogue of \ref{conjOgg2}.
It is supported by the fact that the curve $X_0(\fp)$ has nonzero genus if and only if $\deg \fp \geq 3$ (and this genus is then at least $2$), and by the list of all ideals $\fn$ of $A$ with $\deg \fn\geq 3$ such that $Y_0(\fn)$ has a $F$-rational CM point; cf. \cite{SchweizerMZ}.

\begin{conj*}{TDM$^+$} \label{conjSchweizer2}
Let $C = 3$ if $q\neq 3$, and $C = 4$ if $q = 3$. If $\fp$ is a prime of $A$ of degree~$\geq C$, there exists no Drinfeld $A$-module $\phi$ of rank~$2$ over $F$ with an $F$-rational $A$-submodule of $\phi[\fp]$ isomorphic to $A/\fp$. 
Equivalently, if $\fp$ is a prime of $A$ of degree~$\geq C$, the Drinfeld modular curve $Y_0(\fp)$ has no $F$-rational points.
\end{conj*}

There is also an analogue of Conjecture~\ref{conjUBC} for Drinfeld modules of arbitrary rank $r$. If $r\geq 2$, it is reminiscent of the uniform boundedness conjecture for the torsion of abelian varieties over a number field.
\begin{conj*}{UBC-DM}[{\cite[Conjecture 2]{PoonenUBT}}] \label{conjPoonen}
Fix $q$ as well as $r\geq 1$ and $d\geq 1$. There is a uniform bound on $\#(^\phi L)_\text{tors}$ as $L$ ranges over extensions of $F$ of degree $\leq d$, and $\phi$ ranges over rank-$r$ Drinfeld $A$-modules over $L$. In the rank $r = 2$ case, this is equivalent to saying that there exists a constant $C > 0$ such that if $\fn$ is an ideal of $A$ with $\deg\fn \geq C$ and $L$ is an extension of $F$ of degree $\leq d$, then the Drinfeld modular curve $Y_1(\fn)$ has no $L$-rational points.
\end{conj*}

\begin{rem}
\begin{enumerate}
    \item The dependence of the universal bound on the rank $r$ in Conjecture \ref{conjPoonen} cannot be avoided. Indeed, let $V\subset F$ be an arbitrary $r$-dimensional $\F_q$-vector subspace of $F$, 
    and let $\phi_T(x)=Tx\prod_{0\neq v\in V}(1-x/v)$. Since $\phi_T(x)$ is an $\F_q$-linear polynomial of degree $q^r$, it defines a Drinfeld module over $F$ of rank $r$. On the other hand, by construction, $\phi[T]\cong (A/TA)^r$ is rational over $F$.  
 \item This conjecture in \cite{{PoonenUBT}} is stated for more general rings $A$, namely those of regular functions on the affine curve obtained by removing a closed point from
a nonsingular projective curve over $\F_q$. This more general conjecture reduces to \ref{conjPoonen} for $\F_q[T]$ since any Drinfeld $A$-module can be considered as a Drinfeld $\F_q[T]$-module of higher rank. 
An earlier formulation of the conjecture can be found in Probl\`{e}me~3 of \cite{Denis1995}. 
\item For rank-$1$ Drinfeld modules, Poonen \cite{PoonenUBT} gave several proofs of this conjecture, including an explicit bound on the size of the torsion. (For example, $\# ({^\phi}F)_\tor\leq q$ if $\phi$ has rank $1$ and $q\neq 2$.)  See also Rosen \cite{RosenFG} and Schweizer \cite{SchweizerMZ} for other proofs or improvements in this case. A fact essential for the proofs is that any rank-$1$ Drinfeld module defined over $L$ becomes isomorphic to the Carlitz module over an extension of $L$ of degree $\leq q-1$.  
\item 

It is possible to reformulate Conjectures \ref{conjSchweizer1} and \ref{conjPoonen} in an 
equivalent, elementary form, which makes no reference to Drinfeld modules; cf. \cite{PoonenClassifPreperiodic,Schweizer-Periodicpoints}. 
Using terminology from arithmetic dynamics, we say that $\alpha\in L$ is a \textit{preperiodic point} for the polynomial 
$f(x)\in L[x]$ if the set $\{\alpha, f(\alpha), f(f(\alpha)), f(f(f(\alpha))),\dots\}$ is finite. Note that for a Drinfeld module  $\phi$ over $L$, the set of preperiodic points of $\phi_T(x)$ is exactly $({^\phi}L)_\tor$. Thus, Conjecture \ref{conjSchweizer1} can be stated as saying that for any $f(x)=Tx+gx^q+\Delta x^{q^2}\in F[x]$, with $\Delta\neq 0$, the set 
of preperiodic points in $F$ has cardinality at most $q^2$. Also, Ingram \cite{Ingram} has proposed a perspective on Conjecture~\ref{conjPoonen} via the adelic filled Julia set attached the Drinfeld module when viewed as a dynamical system, with applications to certain families of rank-$r$ Drinfeld modules.
 \item A version of Conjecture \ref{conjUBC} for elliptic curves over function fields is 
        relatively easy to prove, and was already done by Levine \cite{Levin} in 1968. Similarly, let $L$ be a field of transcendence degree $1$ over $F^\alg$, 
        and let $\phi$ be a Drinfeld module of rank $2$ over $L$ such that $j(\phi)\not\in F^\alg$. Schweizer proved in \cite{Schweizer2004} that 
        $$
\# ({^\phi}L)_\tor\leq (\gamma^2(q^2+1)(q+1))^{q/(q-1)},
        $$
        where $\gamma$ is the gonality of $L$. 
\end{enumerate}
\end{rem}

In rank $2$, Conjecture~\ref{conjPoonen} is currently open in general but progress has been made. Recall that a first and important step for the proof of Conjecture~\ref{conjUBC} was Manin's result \cite{Manin}, namely a uniform bound on the $p$-primary torsion of elliptic curves over a number field $L$, depending on $p$ and $L$. Its analogue for Drinfeld $A$-modules of rank~$2$ has been established by Poonen \cite{PoonenUBT}. 
The following theorem of Schweizer improves this result by making the constant depend only on the prime of $A$ and the degree of $L$ (see also \cite{CKK2015} for a generalization to more general function fields than~$F$):
\begin{thm} [{\cite[Theorem 2.4]{SchweizerMZ}}]\label{thm-schweizeruniformpprimary}
Fix $q$, as well as $d\geq 1$ and a prime $\fp$ of $A$. As $L$ ranges over all extensions of $F$ with $[L:F] \leq  d$ and $\phi$ ranges over all rank-$2$ Drinfeld $A$-modules over $L$, there is a uniform bound on the size of the $\fp$-primary part of $(^\phi L)_\tor$.
\end{thm}
As a consequence, Conjecture~\ref{conjPoonen} in the rank-$2$ case is reduced to the problem of bounding uniformly the number of primes $\fp$ in the primary decomposition of the torsion, equivalently to proving that for any $d\geq 1$, there exists a constant $C > 0$ such that if $\fp$ is a prime of $A$ with $\deg\fp \geq C$ and $L$ is an extension of $F$ of degree $\leq d$, then $Y_1(\fp)$ has no $L$-rational points. 

Recently and for arbitrary rank,  Ishii proved a version of Manin's uniformity result on the $\fp$-primary torsion for one-dimensional families of Drinfeld modules over a finitely generated extension of $F$, which is the analogue of a result of Cadoret and Tamagawa \cite{CaTa} for $1$-dimensional families of abelian varieties. More precisely, we have:
\begin{thm}[{\cite[Theorem 1.1]{ishii_-primary_2024}}]
Let $\fp$ be a prime of $A$, $L$ a finitely generated extension of $F$, $S$ a one-dimensional scheme which is of finite type over $L$, and $\phi$ a Drinfeld $A$-module over $S$. Then there exists an integer $N := N(\phi, S, L, \fp)\geq 0$ such that $\phi_s[\fp^\infty](L)\subset \phi_s[\fp^N](L)$ for every $s\in S(L)$.
\end{thm}
Besides this result, not much is currently known towards Conjecture~\ref{conjPoonen} in ranks $r\geq 3$.

We make a few comments about the analogue of Conjecture \ref{conjSUQ}. In a series of papers culminating in \cite{PR}, Richard Pink and his collaborators proved the following analogue of Serre's Open Image Theorem:  

\begin{thm}[Open Image Theorem]\label{thmOIT} Let $\phi$ be a Drinfeld $A$-module of rank $r$ over a finite extension $L$ of $F$. Assume $\End_{L^\sep}(\phi)=A$. Then there is a constant $N(\phi, L)$ depending only of $\phi$ and $L$ such that 
$$
[\GL_r(A/\fn):\rho_{\phi, \fn}(G_L)]\leq N(\phi, L) \quad \text{for all nonzero}\ \fn\lhd A. 
$$
\end{thm}

The analogue of Conjecture \ref{conjSUQ} is the statement that the bound $N(\phi, L)$ 
in Theorem \ref{thmOIT} can be made uniform, i.e., independent of $\phi$. 
For $r=1$ this not hard to prove, but it remains a major open question for $r\geq 2$. 

\begin{rem}
\begin{enumerate}
    \item Let $\phi$ be the rank-$2$ Drinfeld module over $F$ defined by $\phi_T=Tx+x^q-\fp x^{q^2}$. 
    Using the Weil pairing for Drinfeld modules, it is possible to show that 
    $\rho_{\phi, \fp}$ is not surjective if $\fp$ is a prime of odd degree and $q\equiv 1\Mod{4}$, 
    so the direct analogue of Conjecture \ref{conjSUQ} does not hold; see \cite[$\S$4.2]{ChenThesis}. 
    \item In \cite{ChenThesis}, extending a construction of Zywina for $r=2$, Chen proved that 
    for the rank-$r$ Drinfeld module over $F$ defined by $\phi_T=Tx+x^{q^{r-1}}+T^{q-1}x^{q^r}$ the representations $\rho_{\phi, \fn}\colon G_F\to \GL_r(A/\fn)$ are surjective for all $\fn\neq 0$, assuming $r$ is prime, $q\equiv 1\Mod{r}$, and the characteristic of $F$ is sufficiently large compared to $r$. 
\end{enumerate}
\end{rem}

\smallskip
We now turn to known results on Conjectures~\ref{conjSchweizer1} and \ref{conjSchweizer2}, some of which have consequences for Conjecture \ref{conjPoonen}. In $2010$, P\'al made a major progress by proving Conjecture~\ref{conjSchweizer2} for the case $q=2$, i.e., the field $F = \F_2(T)$.
\begin{thm}[{\cite[Theorem 1.2]{PalCrelle}}]\label{thm-pal}
Assume that $q=2$. If $\fp$ is any prime of $A$ with $\deg \fp\geq 3$ then $Y_0(\fp)$ has no $F$-rational points.
\end{thm}
He derives that Conjecture \ref{conjSchweizer1} and Conjecture \ref{conjPoonen} are verified for $q = 2$ and $L = F = \F_2(T)$; cf. Theorems~1.4 and 1.6 in \cite{PalCrelle}.
In a different direction and more recently,  Ishii made another progress towards Conjecture~\ref{conjSchweizer2} for arbitrary~$q$ and levels of small degree.
\begin{thm}[{\cite[Theorem 0.4]{Ishii-isogenychar}}]\label{thm-ishiideg4}
Let $\fp$ be a prime of $A$ of degree $4$. Then $Y_0(\fp)$ has no $F$-rational points.
\end{thm}

\begin{rem}Even though Conjecture~\ref{conjSchweizer2} is not formulated for composite levels, we mention additional results of Schweizer for small degree: the curve $Y_0(\fn)$ has no $F$-rational point when 
$\fn \in \{T(T^2 + T + 1), T^3, T^2(T + 1)\}$ in $\F_2[T]$, $\fn \in \{T(T - 1)(T + 1), T^2(T - 1)\}$ in $\F_3[T]$, and when $\fn$ is the product of three distinct linear factors in $\F_4[T]$; cf. \cite{Schweizer2004, Schweizer2011}.
\end{rem}

For Conjecture~\ref{conjSchweizer1}, Schweizer proved that we always have $\deg \fm \leq 1$ in $(^\phi F)_\tor\cong A/\fm\oplus A/\fn$; cf. Proposition~4.4 in \cite{SchweizerMZ}. 
Moreover since any $F$-rational point on $Y_1(\fp)$ naturally provides an $F$-rational point on $Y_0(\fp)$, Theorems~\ref{thm-pal} and \ref{thm-ishiideg4} remain valid for the curve $Y_1(\fp)$.

\begin{rem}If $q=2$, it is also known that the curve $Y_1((T^2+T+1)^2)$ with composite level has no $F$-rational points; cf.  \cite[Theorem~10.8]{PalCrelle}.
\end{rem}

For finite extensions of~$F$, Armana has obtained partial or conditional results towards Conjectures~\ref{conjSchweizer1} and \ref{conjPoonen}. The first one focuses on levels of small degree.
\begin{thm} [{\cite[Theorem 7.5]{ArmanaANT}}]\label{thm-armana1}
Let $\fp$ be a prime of $A$.
\begin{itemize}
    \item[(i)] If $\fp$ is of degree $3$, the curve $Y_1(\fp)$ has no $L$-rational point for any extension $L/F$ of degree $\leq 2$.
    \item[(ii)]\label{thm-armana1-ii} Suppose $\fp$ is of degree $4$. Let $d\geq 1$ with  $d = 1$ if $q$ is arbitrary, $d = 2$ if $q = 5$, and $d = 3$ if $q \geq 7$. Then the curve $Y_1(\fp)$ has no $L$-rational point for any extension $L/F$ of degree $\leq d$.
    \item[(iii)] Let $\fp$ be a prime of $A$ such that 
    for any normalized Hecke eigenform~$f \in \cH_{0}(\fp,\C)$, we have $\ord_{s=1}L(f,s) \leq 1$. Then the curve $Y_1(\fp)$ has no $L$-rational points for any extension $L/F$ of degree $< \min (\deg \fp, |\fp|/(2(q^2+1)(q+1)))$. 
\end{itemize}
\end{thm}
Given a prime $\fp$, the assumption in (iii) on the vanishing order of $L$-series at the central value is not always satisfied. However, if one believes in the philosophy that elliptic curves over~$\Q$ with rank greater than $1$ are expected to be rare and in its counterpart for elliptic curves over $F$, we should get many examples of pairs $(q,\fp$) to which statement (iii) applies. 
For instance it applies to curves $Y_1(\fp)$ for all prime levels $\fp$ of degree~$5$ in $\F_2[T]$ with $L=F$, and for $\fp=T^5-T^4-T^2-1 \in \F_3[T]$ with $[L:F]\leq 3$. There is also a similar result for $F$-rational points on $Y_0(\fp)$ when $q\geq 5$, under the same hypothesis as in (iii); cf. \cite[Theorem~7.8]{ArmanaANT}. 

The second result is conditional but with a more general conclusion towards Conjecture~\ref{conjSchweizer1}.
To state it we need some notation. Similarly to classical modular forms, it is possible to develop a theory of algebraic Drinfeld modular forms of weight~$2$ using sections of the sheaf of relative differentials on $X_0(\fp)$. Let $\mathcal{M}_{\fp}$ be the $A[1/\fp]$-module of doubly cuspidal algebraic Drinfeld modular forms of weight~$2$ and type~$1$ for $\Gamma_0(\fp)$ and let $\T_{\fp}$ be the Hecke algebra generated by all Hecke operators acting on $\mathcal{M}_{\fp}$. For a prime $\fl$ of $A$, let $\F_\fl = A/\fl$ and $\mathcal{M}_{\fp}(\F_\fl) = \mathcal{M}_\fp\otimes_{A[1/\fp]} \F_\fl$. By extending the results of Section~\ref{subsection-DMF}, any $f \in \mathcal{M}_{\fp}(\F_\fl)$ has an expansion at the cusp $[\infty]$ of the form $\sum_{i\geq 0} b_i(f) t^{1+i(q-1)}$ with coefficients $b_i(f) \in \F_\fl$. In the Hecke algebra~$\T_{\fp}$, let $I_e$ (resp. $\widetilde{I_e}$) be the annihilator ideal  analogue of the winding element (resp. the winding element modulo the characteristic of $F$); cf. \cite{ArmanaANT} for the definitions. Let $\mathcal{M}_{\fp}(\F_\fl)[\widetilde{I_e}]$ be the subspace of $\mathcal{M}_{\fp}(\F_\fl)$ annihilated by $\widetilde{I_e}$.

\begin{thm} [{\cite[Theorem 1.5]{ArmanaANT}}]\label{thm-armana2}
Let $\fp$ be a prime of degree $\geq 3$. Suppose there exist:
\begin{itemize}
    \item[(i)] A saturated ideal $I$ of $\T_{\fp}$, with annihilator denoted $\hat{I}$, satisfying $I_e \subset I \subset \tilde{I_e}$ and $\hat{I} + \tilde{I_e} = \T_{\fp}$.
    \item[(ii)] A prime $\fl$ of $A$ of degree $1$ such that the $\F_\fl$-linear map
    \[ (\T_{\fp}/\tilde{I_e}) \otimes_{\Z} \F_\fl \longrightarrow \Hom(\mathcal{M}_{\fp}(\F_\fl)[\tilde{I_e}], \F_\fl), \]
    which maps $u \in \T_{\fp}$ to the linear form $f \mapsto b_1(f|u)$, is an isomorphism.
\end{itemize}
Then:
\begin{itemize}
    \item[1)] If $\deg \fp \geq \max(q + 1, 5)$, the curve $Y_1(\fp)$ has no $L$-rational point for any extension $L/F$ of degree $\leq q$.
    \item[2)] Conjecture 
    ~\ref{conjSchweizer1} is true for the ideal $\fp$, namely $Y_1(\fp)$ has no $F$-rational points.
\end{itemize}
\end{thm}
The assumption (i) on the existence of the ideal $I$ is likely of technical nature. The assumption (ii) is deeper: it is related to complications that arise when adapting Mazur's formal immersion argument to the Drinfeld setting. This will be discussed in Subsection~\ref{ssProofs}.


\subsection{Outline of the proofs}\label{ssProofs}

We review the main ideas behind the proofs of Theorems~\ref{thm-schweizeruniformpprimary}, \ref{thm-pal}, \ref{thm-ishiideg4}, \ref{thm-armana1} and \ref{thm-armana2}, which are mostly inspired by work of Manin and Mazur, with a focus on the differences from the classical setting and complications that arise.

\smallskip
Schweizer's uniform bound for the $\fp$-primary torsion, Theorem~\ref{thm-schweizeruniformpprimary}, is an analogue of a theorem of Kamienny and Mazur \cite{KM95}, itself a stronger version of Manin's theorem \cite{Manin} for the $p$-primary part of the torsion of elliptic curves over $\Q$. The proof of Schweizer, after Poonen \cite{PoonenUBT}, follows the same approach. The key ingredient essentially states that for any curve~$C$ over a global field $K$, the $d$-fold symmetric power $C^{(d)}$ has only finitely many $K$-rational points if $C$ does not admit a $K$-rational covering of $\mathbb{P}^1_{K}$ of degree $\leq 2d$. When $K$ is a number field, this is a theorem of Frey which derives from Faltings's theorem on rational points of subvarieties of abelian varieties. When $K$ is a function field, Schweizer has obtained a similar criterion from the Mordell-Lang conjecture for abelian varieties over function fields, proved by Hrushovski \cite{Hrushovski96}. He applied it to the Drinfeld modular curves $X_0(\fp^e)$ for primes~$\fp$ to obtain Theorem~\ref{thm-schweizeruniformpprimary}.

\smallskip
To discuss the proofs of Theorems~\ref{thm-pal}, \ref{thm-ishiideg4}, \ref{thm-armana1} and \ref{thm-armana2}, we need to recall Mazur's approach from \cite{Mazur1978} for the study of $\Q$-rational points on the classical modular curve $X_0(p)$, $p$ prime (see also \cite{BMOgg} in this volume and the overview \cite{EdixhovenBourbaki95} for additional details). Handling points defined over a number field $K$ of degree~$d$ requires Kamienny's generalized approach using the $d$-fold symmetric power of $X_0(p)$. For simplicity we focus here only on $d=1$. 

Let $P=(E,C_p) \in X_0(p)(\Q)$ be a non-cuspidal point for a prime number $p$. Assume that the genus of $X_0(p)$ is positive. The objective of Mazur's formal immersion argument is to prove that the elliptic curve $E$ has potentially good reduction at all primes $\ell\neq 2p$; cf. Corollary~4.4 in \cite{Mazur1978}. Consider the Abel-Jacobi map
$$ \begin{array}{rcl} X_0(p) & \longrightarrow & J_0(p) \\ Q & \longmapsto & (Q)-(\infty)\end{array}$$
where $J_0(p)$ is the Jacobian variety of $X_0(p)$. The main ingredient is an optimal quotient $A$ of $J_0(p)$ defined over $\Q$ which satisfies the following two properties:
\begin{enumerate}
	\item[(1)] The Mordell-Weil group $A(\Q)$ is finite;
	\item[(2)] Let $\varphi:X_0(p) \to J_0(p) \to A$, which extends over $\Z$ to $\varphi:X_0(p)_{\mathrm{sm}} \to \mathcal{A}$ where $X_0(p)_{\mathrm{sm}}$ is the largest open subset of $X_0(p)$ smooth over $\Z$ and $\mathcal{A}$ is the N\'{e}ron model of $A$ over $\Z$. When looking at the fibres at any prime $\ell$ such that $\ell \nmid 2p$, the map $\varphi_{\ell} : X_0(p)_{\F_\ell} \to \mathcal{A}_{\F_\ell}$ is a formal immersion along the cuspidal section  $[\infty]$.
\end{enumerate}
Optimal quotients satisfying (1) exist: the idea is to construct them with the property that their Hasse-Weil $L$-function does not vanish at the central value and deduce that their Mordell-Weil group is finite, in the spirit of the Birch and Swinnerton-Dyer conjecture. For $A$, one may take Mazur's Eisenstein quotient \cite{Mazur1977} or Merel's winding quotient \cite{Merel} (the latter one being the largest quotient of $J_0(p)$ satisfying this property). 

Condition~(2) is equivalent to the surjectivity of  corresponding map $\varphi_{\ell}^*:\Cot_0(A_{\F_\fl}) \to \Cot_\infty(X_0(p)_{\F_\ell})$ on the cotangent spaces, equivalently here that $\varphi_{\ell}^*$ is nonzero. It is instructive to reformulate this condition for the Abel-Jacobi map $X_0(p)_{\mathrm{sm}} \to \mathcal{J}_0(p)$ over $\Z$ in terms of $q$-expansions of cusps forms:  on the cotangent spaces at $[\infty]$, it is  simply the map
$$\begin{array}{ccc}
	S_2(\Gamma_0(p),\Z) & \longrightarrow & \Z \\
	\sum_{n\geq 1} a_n q^n & \longmapsto & a_1
\end{array}$$
where $S_2(\Gamma_0(p),\Z)$ is the space of cusp forms of weight $2$ with coefficients in $\Z$. This map is nonzero, as a consequence of the classical formula
\begin{equation}\label{eq-a1Tnf}
	\forall n\geq 1,\quad	a_n(f)=a_1(f|T_n)
 \end{equation}
 for the action of the Hecke operators $(T_n)_{n\geq 1}$ on cusp forms. For the map $\varphi_{\ell}^*$, there is a similar description involving cusp forms annihilated 
 by the Eisenstein ideal, resp. the winding ideal of the Hecke algebra. One then uses a Hecke eigenvector in $\Cot_0(A_{\F_\fl})$, which necessarily has $a_1\neq 0$. This ensures Condition~(2) is satisfied. 

Using (1) and (2), Mazur was able to prove that the elliptic curve $E$ necessarily has potentially good reduction at all primes $\ell \neq 2p$ as follows. If $E$ has potentially multiplicative reduction at~$\ell$, the point $P$ will specialize on $X_0(p)$ at $\ell$ to one of the cusps. By applying the Atkin-Lehner involution $w_p$, we may assume that it specializes to $[\infty]$, therefore the point $\varphi(P)$ specializes to $0$ at $\ell$. By the formal immersion property (2) for $\varphi$ at $\ell$, we also have $\varphi(P)\neq 0$. By (1) we also know that $\varphi(P)$ is a torsion point. This contradicts a specialization lemma for points of finite order in group schemes.

The last part of Mazur's proofs of Conjectures~\ref{conjOgg1} and \ref{conjOgg2} is to discard the cases where $E$ does not have potentially good reduction at~$\ell$ when $p$ is large enough. If the point $P$ on $X_0(p)$ comes from a torsion point of order~$p$ on $E$, this follows from known bounds for the order of the specialization of a torsion point at a prime $\ell$ with $p \neq \ell$, when $E$ has good or additive reduction at $\ell$. If $P$ comes from a rational cyclic subgroup $C_p$ of order~$p$, Mazur studies the \emph{isogeny character}, i.e., the representation $\Gal(\Q^\alg/\Q) \to \GL_1(\Z/p\Z)$ coming from the natural Galois action on $C_p$. Constraints coming from the existence of this isogeny character with the Riemann hypothesis for the reduction of~$E$ modulo $\ell \notin\{ 2,p \}$ ultimately provide a finite list of possible values for $p$.

\smallskip
In this strategy, most of the steps adapt more or less easily to the Drinfeld modular curve $X_0(\fp)$. As an optimal quotient $A$ of the Jacobian $J_0(\fp)$, one may take the Eisenstein quotient introduced by Tamagawa using the Eisenstein ideal $\fE(\fp)$, cf. Definition~\ref{defEI}; another possibility is the winding quotient of \cite{PalCrelle,ArmanaANT}. In both cases their Mordell-Weil group are finite, by Schneider's inequality in the Birch and Swinnerton-Dyer conjecture for abelian varieties over function fields \cite{SchneiderBSD}, so there is an analogue of condition~(1). The main obstacle, however, resides in the function field counterpart of the formal immersion property in condition~(2). Indeed the action of Hecke operators on the expansion of Drinfeld modular forms is not well understood and a formula similar to~\eqref{eq-a1Tnf} is lacking (on this topic see also \cite{Armana2011}). This is the reason why Theorems~\ref{thm-pal}, \ref{thm-ishiideg4}, \ref{thm-armana1} and \ref{thm-armana2} do not provide complete answers to the analogues of Ogg's conjectures. This also explains the assumption~(ii) in Theorem~\ref{thm-armana2}: if one removes the ideal $\widetilde{I_e}$ and the prime place $\fl$ from its formulation, the assumption becomes  equivalent to the perfectness of the $\Ci$-pairing between the space of Drinfeld modular forms $M_{2,1}^{0,0}(\fn)$ and its Hecke algebra, defined  by $(f,u) \mapsto b_1(f|u)$. If true, it would imply a ``multiplicity one" result in $M_{2,1}^{0,0}(\fn)$ which at the moment is open in general (it is known to fail for other congruence subgroups; cf. Böckle \cite[Example~15.4]{BoeckleES} for $\Gamma_1(T)$ and Hattori \cite[Theorem~1]{HattoriTriviality} for $\Gamma_1(T^n)$ -- these are strong indications that multiplicity one could fail for $\Gamma_0(\fn)$). Moreover Armana has recently established that this pairing is not perfect in a quite general case, namely when $\fn$ is prime of degree~$\geq 5$; cf. \cite{Armanapairing}. Although it is not enough to show that the version of this perfectness stated in the assumption~(ii) is not satisfied, it confirms the severity of the obstruction that arises when adapting Mazur's method.

Because of this obstruction, the unconditional statements that we currently know, such as Theorems~\ref{thm-pal}, \ref{thm-ishiideg4} and \ref{thm-armana1}, employ workarounds or variants of the formal immersion property at $\ell$. They are essential of two types:
\begin{itemize}
\item Following an idea of Merel and Parent for classical modular curves, P\'{a}l establishes a variant of the formal immersion property at the place $\infty$ of $F$, instead of a finite place~$\fl$ of $A$; cf. Proposition~7.14 in \cite{PalCrelle}. To prove this variant, he constructs a regular model of the curve $X_0(\fp)$ at $\infty$, uses the incidence graph of the fiber of this model and harmonic cochains on this graph. The study of isogeny characters is replaced with an ad hoc argument. P\'{a}l's strategy uses extensively the hypothesis $q=2$ and ultimately provides Theorem~\ref{thm-pal}.
\item 
If $\fp$ has degree $3$ or $4$, or if the prime $\fp$ is such that for any normalized Hecke eigenform~$f \in \cH_{0}(\fp,\C)$, we have $\ord_{s=1}L(f,s) \leq 1$, the situation becomes simpler. The winding quotient is then $J_0(\fp)$ or isogenous to $J_0(\fp)/(1+w_\fp) J_0(\fp)$ where $w_\fp$ denotes the Atkin-Lehner involution. In this situation, Armana is able to avoid the formal immersion property thanks to the fact that the cusp $[\infty]$ is not a Weierstrass point on $X_0(\fp)$ and to a lower bound of Schweizer on the gonality of $X_0(\fp)$; cf. \cite[Proposition~7.6]{ArmanaANT}. These arguments lead to Theorem~\ref{thm-armana1}.
\par Ishii's work \cite{Ishii-isogenychar}, which is 
a general study of isogeny characters coming from rank-$2$ Drinfeld modules as in Mazur's, is also based on this workaround. Using congruences coming from these characters, he obtains several families of conditions which ensure that $Y_0(\fp)$ has no $F$-rational points; cf. \cite[Theorems~0.1, 0.2, 0.3]{Ishii-isogenychar}. The conditions he combines are of three different types: on $q$, on the prime $\fp$ and on the set of finite places where the Drinfeld module has potentially good reduction. Here again, it is not possible to reach the full potential of Mazur's approach because the formal immersion argument is missing in general. Ishii obtains Theorem~\ref{thm-ishiideg4} for rational points on $Y_0(\fp)$ if $\fp$ has degree~$4$ by combining his work on isogeny characters and Armana's replacement for the formal immersion property in this case. 
\end{itemize}


\section{Torsion of the Jacobian of $X_0(\fn)$}\label{sDJT=C}

\subsection{Cuspidal divisor group} 
For a nonzero ideal $\fn\lhd A$, let $X_0(\fn)$ denote the Drinfeld modular curve introduced in Section \ref{sDM}. Its set of cusps is $X_0(\fn)(\C_\infty)-Y_0(\fn)(\C_\infty)$. Let $J_0(\fn)$ be the Jacobian variety of $X_0(\fn)$. Similar to $\cC_N$ appearing in Conjecture \ref{conjCJ-N}, one defines \textit{cuspidal divisor group} $\cC_\fn$ as the subgroup of $J_0(\fn)$ generated by the divisor classes $[c]-[c']$ of differences of all cusps.

One interesting feature of the theory of Drinfeld modules is that Drinfeld modules of rank $r\geq 3$ do not really have classical analogues (Drinfeld modules of rank $1$ are similar to the multiplicative group of a field and Drinfeld modules of rank $2$ are similar to elliptic curves). Such Drinfeld modules can be defined and studied using a single equation $\phi_T$ like elliptic curves, but the Galois representations arising from their torsion points have images in $\GL_r$ and their modular varieties are $(r-1)$ dimensional. Hence, Drinfeld modules of rank $r\geq 3$ are more complicated than elliptic curves but simpler than abelian varieties. This presents the intriguing possibility of proving results for these Drinfeld modules which are known for elliptic curves but perhaps unknown or very hard for abelian varieties. One such result is the finiteness of the cuspidal divisor group of their modular varieties.

For $r\geq 3$, there are different possible compactifications of $Y_0^r(\fn)$ with desirable properties; cf. \cite{Kapranov} and \cite{FKS}. The Satake compactifications of Drinfeld modular varieties were constructed (at different levels of generality and details of proof) by Gekeler \cite{GekelerSatake}, \cite{GekelerDMHR4}, Kapranov \cite{Kapranov}, Pink \cite{Pink}, and H\"aberli \cite{Haberli}. The constructions by Gekeler, H\"aberli, and Kapranov are rigid-analytic, whereas Pink's construction is algebro-geometric. H\"aberli also proved that the analytic and algebraic Satake compactifications give the same variety. 
The Satake compactification 
can be described analytically as follows. For $\boz=(z_1, \dots, z_r)\in \p^{r-1}(\Ci)$, define $d(\boz)=\dim_F(Fz_1+\cdots+Fz_r)$ and $d_\infty(\boz)=\dim_{\Fi}(\Fi z_1+\cdots+\Fi z_r)$. Then $1\leq d_\infty(\boz)\leq d(\boz)\leq r$ and $\Omega^r=\{\boz\mid d_\infty(\boz)=r\}$. More generally, for $1\leq i\leq r$, put 
$$
\Omega^{i, r} = \{\boz\mid d_\infty(\boz)=d(\boz)=i\},
$$
and $\overline{\Omega^r}=\cup_{1\leq i\leq r} \Omega^{i, r}$. The space $\overline{\Omega^r}$ 
is invariant under the action of $\G_0(\fn)$ by linear fractional transformations. 
The quotient 
$X_0^r(\fn)(\C_\infty)=\G_0(\fn)\bs \overline{\Omega^r}$ is a projective connected normal variety over $\Ci$ of dimension $r-1$ 
containing $Y_0^r(\fn)$ as an open subvariety, which has a canonical model over $F$.
The \textit{cusps} of $X_0^r(\fn)$ are the (geometrically) irreducible components of 
$X_0^r(\fn)(\C_\infty)-Y_0^r(\fn)(\C_\infty)$ of dimension $r-2$.  The \textit{cuspidal divisor group} $\cC^r_\fn$ of $X^r_0(\fn)$ is the subgroup of the divisor class group of $X^r_0(\fn)$ generated by the Weil divisors $[c]-[c']$, where $c$ and $c'$ run over the cusps of $X^r_0(\fn)$. The analogue of the result of Manin and Drinfeld in this setting is the following: 

\begin{thm}
    $\cC_\fn^r$ is a finite group.  
\end{thm}
\begin{proof}
    For $r=2$, this was proved by Gekeler in \cite{GekelerCrelle1984}. For $r\geq 3$ this is a result of Kapranov \cite{Kapranov}; see also \cite[Theorem 7.8]{PW-MA} and \cite[Theorem 10.7]{GekelerVII}. 
The general idea of the proof is to construct sufficiently many modular units on $X_0^r(\fn)$, i.e., functions whose divisors are supported on the cusps, so that the subgroup of degree-$0$ divisors of these functions has finite index in the group generated by all $[c]-[c']$. A natural method for constructing modular units uses the Drinfeld discriminant functions $\Delta_{r}(z_1, \dots, z_r)$. Recall that $\Delta_{r}$ does not vanish on $\Omega^r$. One easily shows that $\Delta_{r}(z_1, \dots, z_r)/\Delta_{r}(\fm z_1, \dots, z_r)$ is invariant under $\G_0^r(\fn)$ for any $\fm\mid \fn$, hence defines a modular unit on $X_0^r(\fn)$ whose divisor can be computed from the order of vanishing of $\Delta_{r}(\boz)$ at the cusps. 
\end{proof}

Computing the group structure of $\cC_\fn^r$ is much harder than proving that this group is finite. For classical modular curves an important tool for solving this problem is a result of Ligozat \cite[Proposition 3.2.1]{Ligozat1975}, which gives necessary and sufficient conditions for a function of the form $\prod_{m\mid N}\eta(m z)^{s_m}$, $s_m\in \Z$, to be a modular unit on $X_0(N)$, where $\eta(z)$ is the 
Dedekind eta function. (Recall that $\eta(z)$ is the $24$-th root of the classical discriminant function.) 
However, such a strong result does not seem to hold over function fields; for example, 
$\Delta_r$ only has a $(q-1)$-th root in $\cO(\Omega^r)^\times$, but the analogue of $24$ in this context is  $(q-1)(q^2-1)$. 

Instead, one follows a different strategy, specific to function fields, which was initiated by 
Gekeler in \cite{gekeler_1997}. The key result here is the exact sequence \eqref{eqGvdP}. 
Given a modular unit $u(\boz)$ on $X_0^r(\fn)$, 
to determine the maximal root that can be extracted from $u(\boz)$ (and thus to determine the order of the underlying cuspidal divisor) one applies $\mathrm{dlog}$ to $u(\boz)$ to get a combinatorial $\Z$-valued function $\mathrm{dlog}(u)\in \Har^1(\sB^r, \Z)$. If one is able to compute the value of $\mathrm{dlog}(u)$ on a few well-chosen edges of $\sB^r$, then $u$ has an $m$-th root only if $m$  divides the gcd of these values. Gekeler utilized this strategy in \cite{gekeler_1997} to compute the order of the divisor $[0] - [\infty]$ on $X_0^2(\fn)$, where $\fn$ is either a prime or a square of a prime. Since $[0]$ and $[\infty]$ are the only cusps when $\fn=\fp$ is prime and these cusps are rational, Gekeler's result implies that $\cC^2_\fp$ is cyclic of order 
\begin{equation}\label{eqGekCp}
    \frac{\abs{\fp}-1}{\gcd(q^2-1, \abs{\fp}-1)}. 
\end{equation}
Ho \cite{ho_rational_2024} extended Gekeler's computation to arbitrary prime power levels $\fp^s$. We note that the formula for the order of $[0] - [\infty]$ on $X_0^2(\fp^s)$ when $s \geq 3$ does not specialize to the cases $s = 1$ or $s = 2$.  

In higher ranks the only result so far is the following generalization of \eqref{eqGekCp} proved in \cite{PW-MA}: $\cC^r_\fp$ is cyclic of order 
\begin{equation}\label{eqPapWei}
    \frac{\abs{\fp}^{r-1}-1}{\gcd(q^r-1, \abs{\fp}-1)}. 
\end{equation}

In analogy with $\cC(N)$, define the \textit{rational cuspidal divisor class group} $\cC(\fn) := \cC^2(\fn)$ of $X_0(\fn)$ to be the subgroup of $J_0(\fn)$ generated by the linear equivalence classes of the degree $0$ rational cuspidal divisors on $X_0(\fn)$. The group $\cC(\fn)$ is explicitly computed in the following cases: 
\begin{itemize}
    \item $\deg(\fn)=3$; see \cite{papikian_eisenstein_2016}. 
    \item $\fn=\fp_1\fp_2$ is a product of two distinct primes; see \cite[$\S$6]{papikian_eisenstein_2015}. 
    \item $\fn=\fp_1\cdots \fp_s$ is square-free, but $\cC(\fn)_\ell$ is determined only for $\ell\nmid q-1$; see \cite{papikian_rational_2017}. In this case, $\cC(\fn)_p=0$, where $p$ is the characteristic of $F$. Moreover, $\cC(\fn)_\ell$ naturally decomposes into a direct sum of $2^{s-1}$ cyclic 
    subgroups each of which is an eigenspace with respect to the Atkin-Lehner involutions of $J_0(\fn)$. 
    \item $\fn=\fp^s$ is a prime power; see \cite{ho_rational_2024}. An interesting fact in this case is that $\cC(\fn)$ is ``mostly" $p$-primary and quite large.
\end{itemize}

\begin{rem}
When $\deg(\fn)=3$ or $\fn$ is square-free, all the cusps of $X_0(\fn)$ are rational by \cite[Proposition 6.7]{gekeler_invariants_2001} and \cite[Lemma 3.1]{papikian_eisenstein_2016}, so $\cC(\fn) = \cC_\fn$, but generally $\cC(\fn)$ can be strictly smaller than $\cC_\fn$.
\end{rem}


\subsection{Analogue of Ogg's conjecture} 
Let $J_0(\fn)(F)_\tor$ be the rational torsion subgroup of the Jacobian variety $J_0(\fn)$. By the Lang-N\'eron theorem, $J_0(\fn)(F)_\tor$ is a finite abelian group. Let, as earlier, 
$\cC_\fn$ be the cuspidal subgroup of $J_0(\fn)$, $\cC_\fn(F)$ be the subgroup of rational points on $\cC_\fn$, and $\cC(\fn)$ be the rational cuspidal divisor class group of $X_0(\fn)$. We have $$\cC(\fn) \subseteq \cC_\fn(F) \subseteq J_0(\fn)(F)_\tor.$$ 
The analogue of Conjecture \ref{conjCJ-N} in this setting is the following: 

\begin{conj*}{CJD-$\fn$} \label{conjCJ-fn}
For any nonzero $\fn \in A$, $$\cC(\fn)= \cC_\fn(F)= J_0(\fn)(F)_\tor.$$
\end{conj*}

\begin{rem}
    In view of the formula \eqref{eqPapWei}, which is valid for all $r\geq 2$, 
        one might wonder whether Ogg's conjecture extends to higher rank Drinfeld modular varieties: 
        Is the $F$-rational torsion subgroup of the Picard group of $X_0^r(\fn)$ 
        cuspidal, i.e., generated by the $F$-rational elements of $C^r_\fn$? 
\end{rem}

There is a natural morphism $X_1(\fn)\to X_0(\fn)$, which, by Picard functoriality, induces a 
morphism $J_0(\fn)\to J_1(\fn)$. The kernel of this morphism, denoted $\cS(\fn)$, is finite; cf. 
\cite[$\S$8]{papikian_eisenstein_2015}. This is 
the \textit{Shimura subgroup} in this context.  
Moreover, when $\fn$ is square-free, $\cS(\fn)$ is a $\mu$-type  
\'etale group-scheme over $F$, so its order is coprime to $p$; 
for the proof of this claim, as well as the calculation of the group structure of $\cS(\fn)$, 
we refer to \cite[$\S$8]{papikian_eisenstein_2015}. 
The analogue of Conjecture \ref{conjSJ-p} is the following: 

\begin{conj*}{SJD-$\fp$}\label{conjSJ-fp} For a prime $\fp\in A$, the maximal $\mu$-type subgroup $\cM(\fp)$ of $J_0(\fp)$ is $\cS(\fp)$. 
\end{conj*}

\begin{rem} For general square-free $\fn$, the maximal $\mu$-type subgroup $\cM(\fn)$ of $J_0(\fn)$ 
can be strictly larger than $\cS(\fn)$; we refer to \cite[$\S$8]{papikian_eisenstein_2015} 
for an explicit example. Also, if $\fn$ is not square-free, then $\cS(\fn)$ has a non-trivial 
connected subgroup-scheme. Thus, Conjecture \ref{conjSJ-fp} does not extend to general $\fn$. But we expect, in analogy with Vatsal's result, that $\cM(\fn)_\ell = \cS(\fn)_\ell$ for any square-free $\fn$ and any prime $\ell\nmid (q-1)$. 
\end{rem}

Conjecture \ref{conjCJ-fn} for prime $\fn=\fp$, as well as Conjecture \ref{conjSJ-fp}, were proved by Ambrus P\'al in \cite{PalDoc}. Overall, P\'al's approach is modeled on Mazur's approach, but with some interesting differences that we will explain below. Combining P\'al's approach with some ideas from  \cite{ohta_eisenstein_2014} and \cite{YOO_torsion_2023}, one can prove partial results towards  
Conjecture \ref{conjCJ-fn}:
\begin{itemize}
    \item When $\fn$ is square-free, we have $\cC(\fn)_{\ell} = (J_0(\fn)(F)_\tor)_{\ell}$ for any prime $\ell\nmid q(q-1)$; see \cite{papikian_rational_2017}.
    \item When $\fn=\fp^s$ is a prime power, we have $\cC(\fp^s)_{\ell} = (J_0(\fp^s)(F)_\tor)_{\ell}$ for any prime $\ell\nmid q(q-1)$; see \cite{ho_torsion_2024}.
\end{itemize}

\subsection{Outline of the proofs}
It is instructive to first recall Mazur's strategy for proving Conjecture \ref{conjCJ-p}, up to $2$-primary torsion. Let $\T(p)$ be the Hecke algebra acting on $J_0(p)$ and $\fE(p)\subset \T(p)$ be the Eisenstein ideal generated by the elements $1+\ell-T_\ell$ for all prime $\ell\neq p$ and by $w_p+1$, where $w_p$ is the Atkin-Lehner involution. 
Using properties of modular forms, Mazur proved that $\T(p)/\fE(p)\cong \Z/n\Z$, where $n=(p-1)/\gcd(12, p-1)$. Let $\cJ$ denote the N\'eron model of $J_0(p)$ over $\Z$, and $\cJ_{\F_\ell}$ denote its fibre at $\ell$. Let $\cG_\ell\colonequals \cJ_{\F_\ell}(\overline{\F}_\ell)[\ell^\infty]$ be the \'etale part of the $\ell$-divisible group of $\cJ_{\F_\ell}$. The Hecke algebra $\T(p)$ acts on $\cJ_{\F_\ell}$. By a theorem of Cartier and Serre, for odd $\ell\neq p$ there is an injection 
$$
\cG_\ell[\ell]\otimes_{\F_\ell} \overline{\F}_\ell\longhookrightarrow H^0(X_0(p)_{\F_\ell}, \Omega^1_{X_0(p)_{\F_\ell}})
$$
compatible with the action of $\T(p)$. Suppose $\ell$ divides $n$, and denote $\fP_\ell\lhd \T(p)$ 
the Eisenstein maximal ideal $(\ell, \fE(p))$. From the above injection, one deduces that the kernel 
$\cG_\ell[\fP_\ell]$ is contained in $H^0(X_0(p)_{\F_\ell}, \Omega^1_{X_0(p)_{\F_\ell}})[\fP_\ell]$. 
Now using the duality between the Hecke algebra and the space of cusp forms, one deduces that this latter space is one dimensional over $\T(\fp)/\fP_\ell\cong \F_\ell$, thus the dimension of $\cG_\ell[\fP_\ell]$ is also at most one over $\T(\fp)/\fP_\ell\cong \F_\ell$. This can be extended to show that $\cG_\ell[\fE(p)]$ is a cyclic module over $\T(p)\otimes_{\Z}\Z_\ell$, so $\cG_\ell[\fE(p)]$ 
is isomorphic to a subgroup of $\Z_\ell/n\Z_\ell$. Let $\cT\colonequals J_0(p)(\Q)_\tor$. The Eichler-Shimura congruence relations imply that $\fE(p)$ annihilates $\cT$; hence, if $\ell\mid \#\cT$, then $\ell\mid n$. From the N\'eron mapping property, we obtain an injection $\cT_\ell\hookrightarrow \cG_\ell[\fE(p)]$. Thus, $\# \cT_\ell\leq \Z_\ell/n\Z_\ell$. On the other hand, 
$\Z_\ell/n\Z_\ell\cong \cC(p)_\ell\hookrightarrow \cT_\ell$, so we conclude that $\cT_\ell= \cC(p)_\ell$. 

\vspace{0.1in}

Before proceeding to the proof in the function field case, we need to discuss a few more facts 
about the Jacobians of Drinfeld modular curves. 
The Jacobian $J_0(\fn)$ has a rigid-analytic uniformization over $\Fi$ as a quotient 
of a multiplicative torus by a discrete lattice; this is closely related with the fact that 
$Y_0(\fn)$ has a rigid-analytic uniformization $\G_0(\fn)\bs \Omega^2$, so $X_0(\fn)$ is a Mumford curve. The connection between the analytic uniformization of Mumford curves and their Jacobians was first explicated by Manin and Drinfeld \cite{MDCrelle1973}. In the setting of Drinfeld modular curves this was done by Gekeler and Reversat in \cite{GekelerREVERSAT}. Denote by $\overline{\G_0(\fn)}$
the maximal torsion-free abelian quotient of $\G_0(\fn)$. Using analytic theta functions, Gekeler and Reversat construct a pairing 
\begin{equation}\label{eqGRpair}
    \langle \cdot, \cdot\rangle \colon \overline{\G_0(\fn)}\times \overline{\G_0(\fn)}\to \Fi
\end{equation}
and show that it induces an exact sequence 
\begin{equation}\label{eqGRs}
0\to\overline{\G_0(\fn)}\xrightarrow{\alpha\mapsto \langle \alpha, \cdot\rangle}\Hom(\overline{\G_0(\fn)}, \C_\infty^\times)\to J_0(\fn)(\C_\infty)\to 0. 
\end{equation}
One can define Hecke operators $T_\fp$ as endomorphisms of $\overline{\G_0(\fn)}$ 
in purely group-theoretical terms as some sort of Verlagerung (see \cite[(9.3)]{GekelerREVERSAT}). 
The above exact sequence then becomes compatible with the action of $\T(\fn)$ on its three terms. 
Also, by \cite[(3.3.3)]{GekelerREVERSAT} and \cite{GN}, there is a canonical isomorphism 
$$\overline{\G_0(\fn)}\xrightarrow{\sim}\cH_0(\fn,\Z)
$$
compatible with the action of Hecke operators, so in \eqref{eqGRs} one can replace $\overline{\G_0(\fn)}$ with $\cH_0(\fn,\Z)$. 

Let $\cJ_0(\fn)$ be the N\'eron model of $J_0(\fn)$ over $\p^1_{\F_q}$. Let 
$\cJ_0(\fn)_{\F_\infty}$ be the fibre of $\cJ_0(\fn)$ over the closed point $\infty$ of $\p^1_{\F_q}$, 
let $\cJ_0(\fn)_{\F_\infty}^0$ be the connected component of the identity, and let $\Phi_\infty(\fn)\colonequals \cJ_0(\fn)_{\oF_\infty}/\cJ_0(\fn)_{\oF_\infty}^0$ be the group of 
connected components of $\cJ_0(\fn)_{\F_\infty}$. 

The valuation of the pairing \eqref{eqGRpair},  
$\ord_\infty \langle \cdot, \cdot\rangle \colon \cH_0(\fn,\Z) \times \cH_0(\fn,\Z)\to \Z$, 
is a weighted version of the cycle pairing on 
$\cH_0(\fn,\Z)$ (cf. \cite[Theorem 5.7.1]{GekelerREVERSAT}), and the uniformization sequence 
\eqref{eqGRs} induces the exact sequence  
\begin{equation}\label{eqGpairPhi}
    0\to \cH_0(\fn,\Z)\xrightarrow{f \mapsto \ord_\infty \langle f, \cdot\rangle}\Hom(\cH_0(\fn,\Z), \Z)\to \Phi_\infty(\fn)\to 0,
\end{equation}
compatible with the action of $\T(\fn)$ on its three terms; see \cite[Corollary 2.11]{GekelerJNTB2}. 
(The pairing $\ord_\infty \langle \cdot , \cdot \rangle$ can also be identified with Grothendieck's 
``monodromy pairing", so the sequence \eqref{eqGpairPhi} is a special case of a theorem of Grothendieck from Expos\'e IX in SGA 7.)

  With these preliminaries out of the way, we can now explain P\'al approach to 
  Conjecture~\ref{conjCJ-fn} for prime $\fn=\fp$, which was extended to square-free $\fn$ 
  and prime power $\fn$ in \cite{papikian_rational_2017} and \cite{ho_rational_2024}, respectively. Let $\cT(\fn)\colonequals J_0(\fn)(F)_\tor$. Let $\ell$ be a prime 
  not equal to $p$. The Eichler-Shimura congruence relations imply that $\fE(\fn)$ annihilates $\cT(\fn)_\ell$. On the other hand, by the N\'eron mapping property, there is a canonical 
  injective morphism $\cT(\fn)_\ell \hookrightarrow \cJ_0(\fn)_{\F_\infty}(\F_\infty)$. 
  Since $\cJ_0(\fn)^0_{\F_\infty}$ is a split torus, we have $\cJ_0(\fn)^0_{\F_\infty}(\F_\infty)\cong (\F_q^\times)^g$, where $g=\dim(J_0(\fn))$. Hence, if we assume that $\ell\nmid (q-1)$, then we get 
  an injective homomorphism $\cT(\fn)_\ell\hookrightarrow \Phi_\infty(\fn)_\ell[\fE(\fn)]$, where 
  the latter group is the subgroup of $\Phi_\infty(\fn)_\ell$ annihilated by the Eisenstein ideal. 
  Fix 
  some $n$ large enough so that $\ell^n$ annihilates $\Phi_\infty(\fn)_\ell$. Multiplying the 
  sequence \eqref{eqGpairPhi} by $\ell^n$ and applying the snake lemma, we get an injection 
  $$
\cT(\fn)_\ell \longhookrightarrow \cE_{00}(\fn, \Z/\ell^n\Z),
  $$
  where $\cE_{00}(\fn, \Z/\ell^n\Z)$ is the submodule of $\cH_{00}(\fn, \Z/\ell^n\Z)$ annihilated 
  by $\fE(\fn)$. 
  Using the Fourier expansions of harmonic cochains, one shows that if $f\in \cH_{00}(\fn, \Z/\ell^n\Z)$ 
  is annihilated by the Eisenstein ideal, then it is a scalar multiple of the reduction modulo $\ell^n$ 
  of a $\Z$-valued Eisenstein series (this is a ``multiplicity-one" statement). From this one deduces that $\cE_{00}(\fn, \Z/\ell^n\Z)\cong \Z_\ell/N(\fn)\Z_\ell$, where $N(\fn)$ is essentially the constant Fourier coefficient of an Eisenstein series. 
  The number $N(\fn)$ can be explicitly computed and it matches the size of $\cC(\fn)_\ell$; cf. Example \ref{exampleEis}. Since 
  $\cC(\fn)_\ell\subseteq \cT(\fn)_\ell \subseteq \cE_{00}(\fn, \Z/\ell^n\Z)\cong \Z_\ell/N(\fn)\Z_\ell$, 
  we conclude that $\cC(\fn)_\ell\cong \cT(\fn)_\ell$. Thus, overall, this argument is similar to Mazur's argument but instead of specializing $\cT(\fn)$ into the fibres of $\cJ(\fn)$ at finite primes, one uses the fibre over~$\infty$ (which does not have a direct analogue over $\Q$). 

  For primes $\ell$ dividing $q-1$, proving $\cC(\fp)_\ell=\cT(\fp)_\ell$ is much harder and 
  is closely linked to the fact that $\T(\fp)$ is locally Gorenstein at the prime ideals containing 
  $\fE(\fp)$. This is also a key fact used in the proof of Conjecture \ref{conjSJ-fp} in \cite{PalDoc}. 
  The Gorenstein  property was proved 
  by P\'al in \cite{PalDoc} by adopting Mazur's Eisenstein descent argument. 
  This property implies that $\cT(\fp)_\ell$ and $\cM(\fp)_\ell$ are dual to each other 
  for $\ell\nmid (q-1)$, and $J_0(\fp)[\fE(\fp)]_\ell=\cC(\fp)_\ell\oplus \cS(\fp)_\ell$. 
  When $\ell\mid (q-1)$, the groups $\cC(\fp)_\ell$ and $\cS(\fp)_\ell$ intersect in $J_0(\fp)$ 
  and the proof requires the construction of an auxiliary ``dihedral subgroup". 
  We will not discuss 
  the details of these arguments due to their more technical nature.

  \begin{rem}
      \begin{enumerate}
      \item P\'al also gave a second proof of the Gorenstein property of the localizations of $\T(\fp)$ at Eisenstein primes in \cite{PalIJNT} by adopting an argument of Calegari and Emerton \cite{CE}. In this approach $\T(\fp)$ is identified with a universal deformation ring $R(\fp)$, which is then shown to be generated by a single element over $\Z_\ell$ using cohomological methods. 
          \item If $N=p_1p_2$ is a product of two distinct primes, then the classical Hecke algebra $\T(N)$ is generally not locally Gorenstein at the Eisenstein primes; cf. \cite{YooMZ}. 
          The same is most likely true for $\T(\fn)$. 
          Hence, some genuinely new ideas 
          might be needed to prove $\cC(\fn)_\ell=\cT(\fn)_\ell$ when $\fn$ 
          is not prime and $\ell\mid q-1$. 
          \item The embedding $\cT(\fn)_\ell\hookrightarrow \Phi_\infty(\fn)[\fE(\fn)]$ plays an important role in the argument outlined earlier. It is a well-known fact due to Ribet and Edixhoven that the component groups of the classical modular Jacobian $J_0(N)$ are annihilated by $T_\ell-(\ell+1)$ for all $\ell\nmid N$. Thus, one might wonder whether $\Phi_\infty(\fn)[\fE(\fn)]=\Phi_\infty(\fn)$ and try to estimate the size of $\Phi_\infty(\fn)$. 
          However, it turns out that $\Phi_\infty(\fn)$ is exponentially larger than its subgroup $\Phi_\infty(\fn)[\fE(\fn)]$; cf. \cite{PapikianMJM}. 
      \end{enumerate}
  \end{rem}
  
 Finally, we discuss the $p$-primary rational torsion subgroup $\cT(\fn)_p$ of $J_0(\fn)$. 
 Let $P\in \cT(\fn)_p$ be an element of of order $p$, and let $G\cong \Z/p\Z$ be the 
 constant subgroup-scheme of $J_0(\fn)$ generated by $P$. 
 By the extension property for \'etale points of N\'eron models, $G$ extends 
 to a finite flat subgroup scheme $\cG$ of $\cJ_0(\fn)$. For a closed point $x$ of $\p^1_{\F_q}$, 
 denote the special fibre of $\cG$ at $x$ by $\cG_{\F_x}$. Suppose $J_0(\fn)$ has purely toric reduction at $x$, i.e., $\cJ^0_{\overline{\F}_x}$ is isomorphic to a product of copies of the multiplicative group $\gm_{m, \overline{\F}_x}$. In that case, if $\cG_{\F_x}$ is a subgroup scheme of $\cJ^0_{\overline{\F}_x}$, then $\cG_{\F_x}\cong \mu_p$. Since the Cartier dual of $\mu_p$ is $\Z/p\Z$ and vice versa, the Cartier dual of $\cG$ would have connected generic fibre but \'etale closed fibre, which impossible. Therefore, $\cG_{\F_x}$ is isomorphic to $\Z/p\Z$  and is a subgroup of $\Phi_x(\fn)$. 
 
 \begin{prop}[{\cite[Corollary 7.15]{PalDoc}}]\label{propTpp}
     If $\fp$ is prime, then $\cT(\fp)_p=0$. 
 \end{prop}
 \begin{proof} 
$J_0(\fp)$ has purely toric reduction at $\fp$, as is shown in \cite{GekelerUber} using Deligne-Rapoport type arguments. Moreover, using the structure of $X_0(\fp)_{\F_\fp}$ and Raynaud's methods, Gekeler computed in \cite{GekelerUber} that $\Phi_\fp(\fp)\cong (\abs{\fp}-1)/\gcd(\abs{\fp}-1, q^2-1)$. In particular, $\Phi_\fp(\fp)$ has no $p$-torsion. The claim of 
the proposition now follows from our earlier argument. 
 \end{proof}

The Jacobian $J_0(\fn)$ has purely toric reduction at $\fp\mid \fn$ in a few other cases when $\fn$ is the product of $\fp$ and a square-free polynomial of degree $\leq 2$. The proof of Proposition \ref{propTpp} extends to these cases to show that $\cT(\fn)_p=0$. Generally, we expect that $\cT(\fn)_p=0$ when $\fn$ is square-free, but this is currently open.
In the other extreme, when $\fn=\fp^s$ is a prime power with $s\geq 3$, we expect that $\cT(\fn)$ 
is ``mostly" $p$-primary, as this follows from Conjecture \ref{conjCJ-fn} and the computation of $\cC(\fp^s)_p$ in \cite{ho_rational_2024}; for example, $\cT(T^s)$ must be a $p$-primary group. 

\begin{rem}
    $J_0(\fn)$ has purely toric reduction at $\infty$ for any $\fn$. Unfortunately, the group $\Phi_\infty(\fn)$ usually has non-trivial $p$-torsion, even when $\fn=\fp$ is prime, so $\Phi_\infty(\fn)$ cannot be used to show that $\cT(\fn)_p=0$ for square-free $\fn$. 
    Moreover,  the structure of $\Phi_\infty(\fn)$ depends on the actual prime divisors of $\fn$, and not only on their degrees, so there is no uniform way of describing this group. 
\end{rem}

\begin{defn}
    We say that $\cT(\fn)_p$ is \textit{Eisenstein} if it is annihilated by  
    $\eta_\fp\colonequals T_\fp-(\abs{\fp}+1)$ for all prime $\fp\nmid \fn$. 
\end{defn}

\begin{lem}
If $\fp$ is a prime of good ordinary reduction of $J_0(\fn)$ then $\eta_\fp$ annihilates 
$\cT(\fn)_p$. 
\end{lem}
\begin{proof}
    By assumption, $\cJ_0(\fn)_{\F_\fp}(\overline{\F}_\fp)[p^s] \cong (\Z/p^s\Z)^g$ 
    for all $s\geq 1$, where $g=\dim J_0(\fn)$. This implies that the reduction 
    $\cT(\fn)_p \to \cJ_0(\fn)_{\F_\fp}(\F_\fp)$ is injective. Let $\Frob_\fp$ 
    denote the Frobenius endomorphism of the abelian variety $\cJ_0(\fn)_{\F_\fp}$. 
    Since $\fp\nmid \fn$, the Hecke operator $T_\fp$ satisfies the Eichler-Shimura relation,
    $$
\Frob_\fp^2-T_\fp\cdot \Frob_\fp+\abs{\fp} =0,  
    $$
    in the endomorphism ring of $\cJ_0(\fn)_{\F_\fp}$. Since $\Frob_\fp$ fixes the reduction 
    of $\cT(\fn)_p$, the endomorphism $1-T_\fp+\abs{\fp}$ annihilates this group. As the reduction 
    map commutes with the action of the Hecke algebra, we get that $\eta_\fp$ annihilates $\cT(\fn)_p$. 
\end{proof}

\begin{prop}
    $J_0(\fn)$ has good ordinary reduction at all but finitely many places of $F$.  
\end{prop}
\begin{proof}
    As we discussed, $J_0(\fn)(\Ci)\cong (\Ci^\times)^g/\La$, where $\La$ is a lattice of rank $g$. 
    Thus, $\dim_{\F_p} J_0(\fn)(\Ci)[p] = \dim_{\F_p}\La/p\La=g$, so 
    $J_0(\fn)$ is an ordinary abelian variety over $F$. Now 
    an argument using the Hasse-Witt matrix implies that $\cJ_0(\fn)$ has good ordinary reduction 
    at all but finitely many places of $F$; cf. \cite[Proposition 8.3]{BoeckleCM}. 
\end{proof}

\begin{example}
    \begin{enumerate}
        \item $J_0(T^s)$ has good ordinary reduction at all primes of $A$ different from~$T$; see \cite[Proposition 3.5]{GekelerDuke}. Thus, $\cT(T^s)_p$ is Eisenstein, so can be studied using the machinery of the Eisenstein ideal. Unfortunately, the standard method for computing the index of the Eisenstein ideal in $\T(\fn)$ relies on the Fourier expansions of harmonic cochains, which requires inverting $p$. 
        \item Let $q=2$ and $\fn=T(T^2+T+1)$. In this case, $J_0(\fn)$ is isogenous to a product 
        of two elliptic curves which can be analyzed using their explicit equations. One concludes that 
        $J_0(\fn)$ has supersingular reduction at $T+1$, and good ordinary reduction at any prime 
        of $A$ not equal to $T$, $T+1$, or $T^2+T+1$. In this case, $\cT(\fn)\cong \cC(\fn)\cong \T(\fn)/\fE(\fn)\cong\Z/15\Z$; cf. \cite{papikian_eisenstein_2016}. 
    \end{enumerate}
\end{example}

Let $J_0(\fn)^\new$ be the ``new" quotient of $J_0(\fn)$, i.e., the quotient of $J_0(\fn)$ by its abelian subvariety generated by the images $J_0(\fm)\to J_0(\fn)$ for all $\fm\supsetneq \fn$ under the morphisms induces by the degeneracy maps $X_0(\fn)\to X_0(\fm)$. 

\begin{prop}
Suppose $\fp$ is a place of good non-ordinary reduction for $J_0(\fn)^\new$. If the image of 
$\cT(\fn)_p$ in $J_0(\fn)^\new$ is non-zero, then $\cT(\fn)_p$ is not Eisenstein. 
\end{prop}
\begin{proof}
    If $\cT(\fn)_p$ is Eisenstein and its image in $J_0(\fn)^\new$ is non-zero, then 
    in the quotient $\T(\fn)^\new$ by which $\T(\fn)$ acts on $J_0(\fn)^\new$ there 
    is a proper maximal ideal $\fM$ containing $p$ and the image of $\fE(\fn)$. In particular,  
    $\eta_\fp\in \fM$. On the other hand, because of the non-ordinary reduction assumption, $T_\fp\in \fM$ by \cite[Theorem B.13 (b)]{BoeckleCM}. This implies that $\abs{\fp}+1\in \fM$. Since $p\in \fM$, 
    we get $1\in \fM$, a contradiction. (This type of arguments also appear in the proof of Theorem~1.2 in \cite{PalIJNT}.)
\end{proof}

\bibliographystyle{amsalpha}
\bibliography{bibliography.bib}

\end{document}